\documentclass[11pt,letterpaper,reqno]{amsart}
\usepackage{caption}
\usepackage{xcolor} 
\usepackage[utf8]{inputenc}
\usepackage[T1]{fontenc}
\usepackage{amssymb,amsmath,mathrsfs,amsthm}
\usepackage{enumitem}
\usepackage{mathtools}
\usepackage[margin=2cm]{geometry}
\usepackage{amsfonts}
\usepackage{ dsfont }

 \newcommand{\nnrm}[1]{{\vert\kern-0.25ex\vert\kern-0.25ex\vert #1 
 		\vert\kern-0.25ex\vert\kern-0.25ex\vert}}
\definecolor{ao}{rgb}{0.0, 0.5, 0.0}

\newcommand{\alp}{\alpha}
\newcommand{\bt}{\beta}
\newcommand{\gmm}{\gamma}

\newcommand{\lmb}{\lambda}

\usepackage{esint,comment}

\allowdisplaybreaks
\newtheorem{thm}{Theorem}

\newtheorem{lem}[thm]{Lemma}
 \newtheorem{lemma}[thm]{Lemma}
\newtheorem{prop}[thm]{Proposition}

\theoremstyle{definition}
\newtheorem{defn}[thm]{Definition}
\newtheorem{rem}[thm]{Remark}
\numberwithin{equation}{section}

\def \no#1#2#3 {{\bf #1} (#3), #2.}
\def \eds#1#2#3 {#1, #2, #3.}

\def\R{{\mathbb R}}

\def\d{{\rm d}}

\def\:{{\colon}}

\def\be#1{\begin{equation}\label{#1}}
\def\ee{\end{equation}}

\def\<{\langle}
\def\>{\rangle}
\def\coloneqq{:=}

\newcommand{\norm}[1]{\Vert#1\Vert}

\newcommand{\wo}{\widetilde{\omega}}

\newcommand{\oo}{\overline{\omega}}

\newcommand{\ov}{\overline{v}}

\newcommand{\weta}{\widetilde{\eta}}

\newcommand{\ud}{\mathrm{d}}

\newcommand{\lec}{\lesssim}
\newcommand{\gec}{\gtrsim}
\newcommand{\bs}{\begin{split}}
\newcommand{\essss}{\end{split}}

\renewcommand{\lec}{\lesssim}

\renewcommand{\div}{\operatorname{div}}
 
\newcommand{\eqnb}{\begin{equation}}
\newcommand{\eqne}{\end{equation}}

\renewcommand{\ee}{\mathrm{e}}

\newcommand{\p}{\partial}

\newcommand{\tp}{{\psi}}
\newcommand{\tb}{{\phi}}
\newcommand{\otp}{{\overline{\psi }}}
\newcommand{\otb}{{\overline{\phi}}}
\newcommand{\ttp}{{\Psi }}
\newcommand{\ttb}{{\Phi}}
\newcommand{\lt}{{\widetilde{\lambda}}}
\newcommand{\nt}{{\widetilde{N}}}

\renewcommand{\R}{\mathbb{R}}

\renewcommand{\d}{\mathrm{d}}

\newcommand{\supp}{\operatorname{supp}}

\newcommand{\curl}{\mathrm{curl}}

\newcommand\blfootnote[1]{%
  \begingroup
  \renewcommand\thefootnote{}\footnote{#1}%
  \addtocounter{footnote}{-1}%
  \endgroup
}

\begin{document}
\title[]{Instantaneous continuous loss of Sobolev regularity for the 3D incompressible Euler equation}

\author{In-Jee Jeong,  Luis Mart\'inez-Zoroa,  Wojciech S. O\.za\'nski} 

\maketitle

\blfootnote{\noindent I.-J. Jeong: Department of Mathematical Sciences and RIM, Seoul National University, South Korea, email: injee{\textunderscore}j@snu.ac.kr\\
L.~Mart\'inez-Zoroa: Department of Mathematics and Computer Science, University of Basel, Switzerland, email: luis.martinezzoroa@unibas.ch\\
W.~S.~O\.za\'nski: Department of Mathematics, Florida State University, Tallahassee, FL 32306, USA, email: wozanski@fsu.edu}

\begin{abstract}
 
We prove instantaneous  and  continuous-in-time  loss of supercritical Sobolev regularity for the 3D incompressible Euler equations in $\mathbb{R}^{3}$. Namely, for any $s\in (0,3/2)$ and $\varepsilon >0$, we construct a divergence-free initial vorticity $\omega_0$ defined in $\mathbb{R}^{3}$ satisfying $\| \omega_0 \|_{H^s}\leq \varepsilon$ and there exist $T>0$, $c>0$ and a corresponding local-in-time solution $\omega$ such that, for each $t\in [0,T]$, \begin{equation*}
	\begin{split}
		 \omega (\cdot ,t ) \in {H^{\frac{s-ct}{1+ct}}} \qquad \mbox{and} \qquad  \omega (\cdot ,t ) \not \in {H^\beta }
	\end{split}
\end{equation*} for any $\beta > \frac{s-ct}{1+ct} $. Moreover, $\omega$ is unique among all solutions with initial condition $\omega_0$ which are locally $C^2$ and belong to $C([0,T];L^p )$ for any $p>3 $.

\end{abstract}
\vspace{0.5cm}
\vspace{0.5cm}
\section{Introduction}\label{sec_intro}
We consider the $3$D incompressible Euler equations,
\eqnb\label{3d_euler}
\begin{split}
\p_t v + v\cdot \nabla v + \nabla p &=0, \\
\div\,v &=0.
\end{split}
\eqne
for $x\in \R^3$, $t>0$. Letting $\omega \coloneqq \curl \, v$ we obtain the vorticity formulation
\eqnb\label{3d_euler_vorticity}
\begin{split}
\p_t \omega  + v\cdot \nabla \omega &=\omega \cdot \nabla v,
\end{split}
\eqne
where $v$ is related to $\omega$ via the Biot-Savart law
\eqnb\label{bs}
v[\omega ] (x) \coloneqq \frac{-1}{4\pi }\int_{\R^3} \frac{(x-y )\times \omega (y) }{|x-y|^3}\d y
\eqne
at each time.

The theory of local well-posedness of the Euler equations \eqref{3d_euler} goes back to the works of Kato \cite{Kato}, as well as Kato and Ponce \cite{KP}, who showed that the equations are locally well-posed in the velocity formulation in $H^s$ for any $s >5/2$. The borderline case of $H^{5/2}$ was only settled recently by Bourgain and Li \cite{BL1,BL3}, who showed strong ill-posedness in $H^{5/2}$, see also simplified proofs in \cite{EJ,JK2}. In the vorticity formulation \eqref{3d_euler_vorticity}, this means that there exists a divergence-free $\omega_0\in {H^{3/2}}$ and a unique local-in-time solution $\omega(t)$ {in a suitable class} to the equations \eqref{3d_euler_vorticity} such that $\omega(t) \not \in {H^{3/2}}$ for {arbitrarily small} $t>0$.

In this work, we are interested in $H^{s}$ ill-posedness  phenomena in the strictly supercritical regime, that is for $\omega_0 \in H^s(\mathbb{R}^{3})$ with $0<s<3/2$. {A recent work of Luo \cite{Luo} proved norm inflation for vorticity in $H^s(\mathbb{R}^{3})$ using axisymmetric flows with swirl: for every $s\in (-1,3/2)$ there exists a divergence-free $\omega_0 \in H^s$ such that $\| \omega_0 \|_{H^s}$ is arbitrarily small, and the $H^s$ norm of the unique solution grows arbitrarily fast in time.
Earlier works on 3D supercritical illposedness were mostly based on the so-called $(2+\frac12)$-dimensional flows. Bardos--Titi \cite{BT1,BT2} obtained ill-posedness for $C^\alpha, W^{1,p}$ velocities with $\alp<1$ and any $p$. Moreover, Elgindi and Masmoudi \cite[Proposition~10.1]{EM} proved loss of $L^p$ regularity of vorticity with $p>2$. Jeong and Yoneda \cite{JY1,JY2} obtained ill-posedness for vorticity below $H^{1}$ and connected it with enhanced dissipation for the Navier--Stokes equations. 
}

{In the case of 2D incompressible Euler equations in $\mathbb{R}^{2}$, the critical Sobolev space for the vorticity is $H^{1} = W^{1,2}$. For supercritical $W^{1,p}$ spaces, namely when $p<2$, \cite{JEuler} obtained continuous loss of regularity by proving the existence of $\omega_0\in W^{1,p_0} \cap L^\infty$ for every $p_0 \in (1,2)$ whose unique Yudovich solution $\omega(t)$ belongs to $W^{1,p(t)}$, where $p(t)$ is a continuously decreasing function.} Furthermore, C\'ordoba, Mart\'inez-Zoroa and O\.za\'nski \cite{CMO}  showed instantaneous gap loss of Sobolev regularity for supercritical $H^{s}$ data for the 2D incompressible Euler equations for any $0<s<1$. 

In our first result, we construct a smooth and compactly supported initial data $\omega_0$, which gives a local-in-time solution that grows in $L^{2}$ based Sobolev spaces for $t>0$.

\begin{thm}[Norm inflation from smooth data]\label{thm_black_box}
For every $s \in (0,3/2)$, there exists $T\in (0,1/2]$ such that for every $K\geq 1$, and every sufficiently large $\lambda\geq 1$ there exists $\omega^{(0)}_{K,\lambda} \in C_c^\infty(\mathbb{R}^{3})$ such that there exists a unique smooth solution $\omega_{K,\lambda }$ of the 3D incompressible Euler equations \eqref{3d_euler_vorticity} on time interval $[0,T]$ with
\[
\omega_{K,\lambda } = \tb + \tp
\]
fulfilling 
\eqnb\label{black_box_claims}
\begin{split}
&\supp\, \omega_{K,\lambda } (\cdot , t)  \subset B(0,\lambda^{-1} ),\\
&\| \omega_{K,\lambda } \|_{H^5 } \lec K^{-1} \lambda^{c_1},\\
&\| \omega_{K,\lambda } \|_{L^p } \lec K^{-1} \lambda^{-{C_p}}\quad \text{ for every } p\in [1,\infty ),\\
&\| \tb (\cdot , t) \|_{H^\beta }  \sim K^{-1} \lambda^{c_2(\beta -s)}, \\
&\| \tp (\cdot , t) \|_{H^\beta } \sim K^{-t(1+\beta )/s} \lambda^{c_3(\beta - s) + t(1+\beta )c_4}  (\log \lambda )^{-t(1+\beta )/s}
\end{split}
\eqne
for all $t\in [0,T]$, $\beta \in [0,s]$, where $c_1,c_2,c_3,c_4,C_{p}>0$ are constants depending on $s$ only.
\end{thm}  
{That is, the solution consists of two parts $\phi$ and $\psi$, such that $\phi$ does not vary much in $H^{\bt}$ norms and $\psi$ grows in $H^{\bt}$ norms with time.} Here $B(0,\lmb^{-1})$ is the ball centered at 0 with radius $\lmb^{-1}$. 

In our second result, we employ the construction developed in Theorem~\ref{thm_black_box} to obtain initial vorticity $\omega_0 \in H^s(\mathbb{R}^{3})$, $s\in (0,3/2)$ which gives rise to  a unique local-in-time solution that loses Sobolev regularity in time. We first introduce a notion of solution for such irregular initial data. 

\begin{defn}(Classical solution)\label{defn_sol}
We say that $\omega \colon \R^3 \times [0,T_0] \to \R^3$ is a \emph{classical solution} to the Euler equations \eqref{3d_euler_vorticity} if $\omega \in C_t L_x^p$ for some $p \in (3,\infty )$ and $\omega \in C^1 ([0,T_0] ; C^2 (K))$ for every compact set $K\subset \R^3$.
\end{defn}

We now state our main result.

\begin{thm}[Instantaneous continuous loss of Sobolev regularity of the 3D Euler equation]\label{thm_main}
For every $s\in (0,3/2)$, $\varepsilon >0$ there exists divergence-free $\omega_0$ such that $\| \omega_0 \|_{H^s} \leq \varepsilon$ and there exist $T_0>0$ and a unique classical solution $\omega$ to the Euler equations \eqref{3d_euler_vorticity} on $[0,T_0]$ with initial condition $\omega_{0}$ such that
\[
\omega (\cdot , t) \in H^{\frac{s-\overline{c} t}{1+ \overline{c} t }} \qquad \text{ and } \qquad \omega (\cdot , t) \not \in H^{\beta'} \quad \text{ for all } \beta'>\frac{s-\overline{c} t }{1+ \overline{c} t }
\]
for each $t\in [0,T_0]$, where $\overline{c}>0$ is a constant.
\end{thm}


\subsection{Ideas of the proof}\label{sec_ideas_of_pf}

The proof of Theorems~\ref{thm_black_box}--\ref{thm_main} exploits the mechanism of growth of solution due to a hyperbolic point. In this sense the main idea employed here is similar to the recent result of \cite{CMO_SQG} for the surface quasi-geostrophic equation (SQG). However capturing the growth due to a hyperbolic point is much more challenging the case of the 3D Euler equations \eqref{3d_euler_vorticity}. First of all, the growth scenario must be three-dimensional in the sense that the directions of all components of the vorticity vector must be, in a sense, compatible with the respective directions of high-frequency oscillations (see Figure~\ref{fig1} below). Furthermore, the equation is different: it is not a scalar transport equation, but instead involves transport and stretching of  a vector. In particular, the equation does not preserve the $L^{p}$ norms. Moreover, the vorticity vector  must remain divergence free, and we need to keep control of its growth in all directions.

In order to prove Theorem~\ref{thm_black_box}, we start with an initial condition for the background vorticity of the form

\eqnb\label{id1}
\tb (x,0) \sim \lambda^{\frac32 -s }N^{-s} g(\lambda x) \sin (\lambda N x_2 ) \sin (\lambda Nx_3) e_1,
\eqne
where $g\in C^\infty (\R^3 ; [0,1])$ has compact support.  Note that 
\[ \| \tb ( \cdot , 0)\|_{\dot H^s} =O(1)\]
for all values of the parameters  $\lambda , N \geq  1$, which we will take to be very large. 
The symbol ``$\sim$'' in \eqref{id1} indicates the fact that we have neglected scaling of $\tb (x,0)$ by a factor of $K^{-1}$, which is needed only for using the construction of Theorem~\ref{thm_black_box} as summable building blocks in the proof of the instantaneous loss of regularity (Theorem~\ref{thm_main}), and we also neglected the $e_3$ component of the background vorticity, which does not play any substantial role in the construction, except for guaranteeing the divergence-free condition (otherwise, the $e_3$ component is treated as a lower-order error term in all computations).

The main point of \eqref{id1} is to generate a stable steady hyperbolic flow near the origin, which can then be used to grow a perturbation vorticity $\tp$. We now explain the precise meaning of this.

First we note that we expect the solution $\tb (t)$ of the Euler equations \eqref{3d_euler_vorticity} with initial data \eqref{id1} to remain ``almost'' steady, since $e_1 \sin x_2 \sin x_3$ is a steady state of \eqref{3d_euler}, and since the $\sin$'s in \eqref{id1} have much higher frequency than the cutoff function $g$, the $\sin$'s are expected to determine the leading order dynamics of the solution $\tb$. To be precise, we can denote by $\otb (x) \coloneqq \tb (x,0)$ the \emph{pseudosolution of the background vorticity}, and we can define the \emph{pseudosolution of the velocity of the background vorticity} as the velocity field of $\otb$ in which Biot-Savart operator does not see  $g$, that is
\eqnb\label{id2}
\ov [\otb ] (x) \coloneqq  \frac{\log N}{\sqrt{2} \lambda N } g(\lambda x) \begin{pmatrix}
0 \\
\sin (\lambda N x_2 ) \cos (\lambda N x_3 )\\
-\cos (\lambda N x_2 ) \sin (\lambda N x_3)
\end{pmatrix} .
\eqne
Then $G\coloneqq \ov [\otb ] \cdot \nabla \otb - \otb \cdot \nabla \ov [\otb ]$ can be shown to be, roughly speaking $N$ times smaller than $\otb$ in any $C^{k,\alpha}$ norm, see \eqref{G_error} for details. 

Moreover, the higher frequency under the $\sin$'s in \eqref{id1} also allows us to show that $\ov [\otb ]$ is a good approximation of $v[\tb ]$, so that $\ov [\otb ] - v[\otb ]$ is, roughly speaking, $N$ times smaller than $v[\otb ]$ in any $C^{k,\alpha }$ norm, see Lemma~\ref{lem_vel_err} for details.

We now explain the meaning of $\tb $ being ``stable''. Consider the error $\ttb \coloneqq \tb - \otb $. Then $\ttb $ satisfies a PDE (see \eqref{ttb_eq}), which we need to estimate in $L^\infty$ (actually $C^\alpha$, for some small $\alpha >0$, is used as a more practical choice, but the difficulties are the same). The most difficult terms in the PDE, which include $\ttb$ are $\tb \cdot \nabla v[\ttb ] $ and $\ttb \cdot \nabla \ov [\otb ]$. Both of these terms are of the size 
\[
 \lambda^{\frac32 -s} N^{-s} \| \ttb \|_\infty \sim \| \tb \|_\infty \| \ttb \|_\infty ,
\]
where we assummed that $ \| \ttb \|_\infty < \| \tb \|_\infty $ (so that $\| \tb \|_\infty \lec \| \otb \|_\infty \sim \lambda^{\frac32-s} N^{-s}$) and we used a heuristic that taking a derivative is roughly equivalent to multiplying by $\lambda N$, so that $\| \nabla \ov [\otb ]\|_\infty \sim \lambda^{\frac32-s} N^{-s}$. We also neglected the issue of unboundedness of the singular integral operator $\ttb \mapsto \nabla v[\ttb ]$ in $L^\infty$ to write $\| \nabla v[\ttb ]\|_\infty \lec \| \ttb \|_\infty $ (this issue can be handled by using a logarithm-type nonlocal estimate \eqref{vel_est_general} and some control of $\ttb $ in $C^\alpha $, see \eqref{id1a} for details).
This gives that 

\eqnb\label{id3}
\frac{\d }{\d t} \| \ttb \|_\infty \lec \lambda^{\frac32 -s} N^{-s} \| \ttb \|_\infty + \text{(other terms)},
\eqne
and it shows that the behaviour of $\| \ttb (t)\|_\infty$ will depend exponentially on $\lambda^{\frac32 -s} N^{-s}$. However, 
\[
\lambda^{\frac32 -s} N^{-s} \sim \| \nabla \ov [\otb ] \|_\infty \sim \| \otb \|_\infty,
\]
and thus, due to the desired control of $\ttb$, the same is expected to hold for $\nabla v[\tb ]$  and $\tb$. This means that $\lambda^{\frac32 -s} N^{-s}$ also determines the deformation rate of the particle flow of $v[\tb ]$. We are thus faced with a dichotomy: we want $\lambda^{\frac32 -s} N^{-s}$ to remain under control, so that we can control the error $\ttb$, and we also want it to be large, so that $v[\tb ]$ can cause  large deformations. Since we expect the sizes of all norms of $\tb$ to remain algebraic in $\lambda,N$, the exponential dependence mentioned above suggests that $\exp (\lambda^{\frac32 -s } N^{-s})$ must also be at most algebraic. We thus make a choice of the relation $\lambda $ vs. $N$ as 
\eqnb\label{id4}
\lambda^{\frac32 -s } N^{-s} \sim \log N.
\eqne
Then \eqref{id3} implies that
\eqnb\label{id5}
\| \ttb \|_\infty \lec \text{(other terms)} N^{ct}
\eqne
which is going to be smaller than $\log N \sim \| \tb \|_\infty$, provided the ``$(\text{other terms}$)'' involve a negative power of $N$ and we consider suitably short times. Indeed, one can show that these terms are of the size, roughly speaking $N^{-1}$ (see \eqref{V_and_ttb}), so that \eqref{id5} shows that there exists $T>0$, independent of $\lambda$, such that $\| \ttb \|_\infty \ll \| \tb \|_{\infty}$ on  $[0,T]$, which gives the desired stability of the background vorticity $\tb$.\\

In order to capture the deformation generated by $v[\tb ]$ we now note that the form of the pseudosolution \eqref{id2} of the velocity $v[\tb]$ suggests that 
\eqnb\label{id5a}
\nabla v[\tb ] \sim \begin{pmatrix}
0 & 0 & 0\\
0 & \log N & 0\\
0 & 0 & - \log N
\end{pmatrix}
\eqne
for $t\in [0,T]$, which we make precise in Lemma~\ref{lem_eta} below. This shows that $v[\tb]$ admits a hyperbolic flow in the $x_2-x_3$ plane, which is, roughly speaking, stretching in $x_2$ and squeezing in $x_3$, and we need to come up with a \emph{perturbation vorticity} $\tp$, which will not ``disturb'' the dynamics of the background vorticity $\tb$, and whose dynamics in time will be  determined by   $v[\tb ]$.  

To this end we consider an initial condition  for the perturbation vorticity to live at a much higher frequency than $\tb$ and whose direction ensures cancellation with $\tb$ to the leading order, namely we set 
\eqnb\label{id6}
\tp (x,0) \sim \lt^{\frac32-s} \nt^{-s} g(\lt x) \sin (\lt \nt x_3) e_2,
\eqne
where, similarly to  \eqref{id1}, we neglected the scaling by a factor of $K^{-1}$ and we neglected the $e_1$ component, whose only role is to maintain the divergence-free condition of $\tp$. We will take 
\eqnb\label{id_choice_aB}
\nt = \lt^a,\qquad \lt = \lambda^B
\eqne
for some $a>0$ and some very large $B>1$ (see \eqref{choice_a} and \eqref{choice_B} for the respective definitions). We define the \emph{pseudosolution $\otp$ of the perturbation vorticity} to be a vector field which is advected and stretched by $v[\tb]$, namely that $\otp (x,t)$ solves the PDE
\eqnb\label{id6a}
\p_t \otp + v[\tb ]\cdot \nabla \otp = \otp \cdot \nabla v[\tb ]
\eqne
with initial condition $\otp (x,0) = \tp (x,0)$. Using the Cauchy formula (see \eqref{cauchy_formula} below) we observe that
\eqnb\label{id7}
\otp (x,t) = \nabla \eta (y,t) \tp (y,0),
\eqne
where $y = \eta^{-1}(x,t)$, and $\eta$ denotes the particle trajectory of $v[\tb ]$ (i.e. it is defined by \eqref{new_eta_def}). 

Before discussing the growth of $\otp$, we let $\tp$ denote the function such that $\tb+\tp $ solves the Euler equation \eqref{3d_euler_vorticity} with initial condition $\tb (\cdot ,0) + \tp (\cdot ,0)$, and we   note that the leading order terms in the equation for the \emph{perturbation vorticity error} 
\[
\ttp \coloneqq \tp - \otp
\]
 (see~\eqref{eq_ttp} below)  are
\eqnb\label{id_terms}
\tb \cdot \nabla v[\otp ], \qquad \ttp \cdot \nabla v[\tb ], \qquad  v[\tp] \cdot \nabla \otp , \,\, \otp \cdot \nabla v[ \otp ],   
\eqne
see the terms denoted by $III_3$, $II$, $I$ and $III_1$ (respectively) in \eqref{ttp_eq}. As in the case of the background vorticity error $\ttb$, the most important estimate of the perturbation vorticity error $\ttp$, which we need to control, is $\|\ttp\|_\infty$.\\

 The first term in \eqref{id_terms} represents the influence of the perturbation vorticity $\tp$ onto the background vorticity $\tb$, and it can be controlled by noting that $v[\otp]$ is very small away from $\supp\, \otp$, while on $\supp\,\otp$ one can use the fact that $\tb (0,t)=0$ (due to an assumed odd symmetry with respect to the origin, see~\eqref{tb_symmetry} for details) and  Lipschitz continuity of $\tb$ to obtain an additional factor of ``almost $\lt^{-1}$''. We also need to exploit the gain due to the high oscillations in $x_3$ in $\otp$. Similarly to the background vorticity error $\ttb$ the oscillations make our scenario more stable. On the technical side, this is visible in the perturbation error analysis through a quantitative Riemann-Lebesgue lemma-type trick, which involves integrating the oscillatory term by parts to obtain an additional term $\nt^{-1}$ (see, for example, \eqref{votp_away}). Moreover, it suffices to include oscillations of $\tp$ in one direction only, and the reason for our choice of $x_3$ is explained below. These two ideas make the first term in \eqref{id_terms} manageable, see \eqref{tb_na_v_otp} for details.

The second term in \eqref{id_terms} represents  the influence of the background vorticity $\tb$ onto the perturbation error $\ttp$ and it is under control for the same reason as the background vorticity error $\ttb$, discussed above. Namely,  we have $\| \nabla v [\tb ]\|_\infty \sim \| \nabla \ov [\otb ]  \|_\infty \sim \log N$, which is ``just fine'' for controlling $\| \ttp \|_\infty$, since it gives a power of $N^{ct}$ via a Gronwall estimate.

The last two terms in \eqref{id_terms} represent the  self-interaction of $\tp$ onto itself. These terms might seem the most dangerous. However, in our scenario, there is no need for $\| \otp \|_\infty $ to remain large, unlike for the background vorticity, $\| \otb \|_\infty \sim \log N$. Indeed,  we only need some $H^\beta$ norms of $\otp$ to become large. We thus make a choice of $a$ (recall \eqref{id_choice_aB}) so that
\eqnb\label{id9}
\lt^{\frac32-s} \nt^{-s} \to 0\qquad \text{ as } \lambda \to \infty,
\eqne
i.e. that $a> (3-2s)/2s $. This makes all self-interaction terms negligible, as long as  $B>1$ is sufficiently large, so that any negative power of $\lt$  beats any positive power of $\lambda$ in the resulting estimates. This is sufficient for controlling the self-interactions. We also take $\gamma \in (0,1)$ very small so that, apart from  $\| \tp \|_\infty$ decaying, we also have that $\| \tp \|_{C^\gamma}$ also decays (as $\lambda \to \infty$), which is possible due to  \eqref{id9} by continuity.\\

We also note that, in both estimates of the background error $\ttb$ and the perturbation error $\ttp$, we need to make use of both the vorticity formulation \eqref{3d_euler_vorticity} of the equations and the velocity formulation \eqref{3d_euler}. The velocity formulation is necessary to control the lowest order terms, and estimates from the two formulations must be coupled to obtain effective control, see \eqref{V_and_ttb} and \eqref{ttp_and_W_est} for details. In particular, terms analogous to \eqref{id_terms} appear in the velocity formulation, and become more challenging. To be precise, all estimates at the level of $\| v[\ttp ] \|_\infty$ must be a factor of $\lt \nt $ smaller than the corresponding terms at the level of $\| \ttp \|_\infty$, which is $\nt$ times smaller than what can be expected by scaling analysis. Indeed, $\lt$ is the concentration parameter, which admits natural scaling, while $\nt$ governs the leading order frequency of $\tp$, and so we must be keep a careful track of it in the velocity formulation, see, for example, Lemma~\ref{lem_votp}.\\

Having discussed the control of the error $\ttp $, it now remains to quantify the growth of $\otp$. Recalling the Cauchy formula \eqref{id7}, we see that we need to prove  precise estimates on $\eta$ and $\eta^{-1}$, which we achieve in Lemma~\ref{lem_eta}. For example, by \eqref{id5a},
\[
\nabla \eta  \simeq  \begin{pmatrix}
1 & 0 & 0\\
0 & N^{t/2} & 0\\
0 & 0 & N^{-t/2}
\end{pmatrix} \quad \text{ on }\supp\, \otp(\cdot, 0) \subset B(0, \lt^{-1}),
\]
see \eqref{new_grad_eta} for details. This shows that the hyperbolic flow generated by $v[\tb ]$ gives two simultaneous mechanisms for growth of $\otp$:
\begin{enumerate}
\item Growth due to the transport term $v[\tb ] \cdot \nabla \otp$ in \eqref{id6a}. This is due to the ``$y$'' appearing in $\tp( y,0)$ in \eqref{id7}. To be precise, we expect $\nabla y = (\nabla \eta)^{-1}$ to be diagonal (to the leading order) with entries $1$, $N^{-t/2}$ and $N^{t/2}$ (which we show in \eqref{nabla_y} below). The last diagonal entry is the largest, as expected from the hyperbolic flow, which causes squeezing in $x_3$ (see~Figure~\ref{fig1} for a sketch). This is why the initial condition \eqref{id6} for $\tp$ has the highest frequency in $x_3$. Thus, whenever a $\p_3$ derivative falls on $\sin (\lt \nt y_3)$ in $\tp(\cdot ,0)$ in \eqref{id7}, we expect a factor of $N^{t/2}$, representing the growth due to the transport term. We emphasize that this growth is only visible through the derivatives of $\otp$.
\item Growth due to the vortex stretching term $\otp \cdot \nabla v[\tb ]$. This is due to the stretching factor $\nabla \eta$ appearing in \eqref{id7}. The largest entry of $\nabla \eta$ is the second diagonal entry $N^{t/2}$, which means that the vector $\otp$ will be stretched in the $e_2$ direction (see also Figure~\ref{fig1}). This is why we make $\tp(\cdot ,0)$ point only  in the $e_2$ direction (i.e. to the leading order, as mentioned below \eqref{id6}). We emphasize that this growth occurs at the level of the magnitude, i.e. is visible without taking any derivatives of $\otp$.
\end{enumerate}
\begin{center}
 \includegraphics[width=5cm]{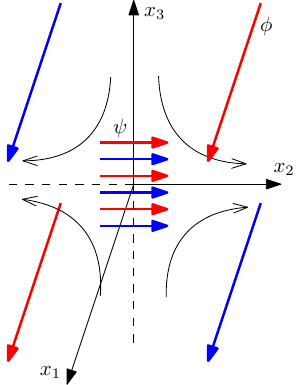}
 \end{center}
 \nopagebreak
 \captionsetup{width=.8\linewidth}
  \captionof{figure}{A sketch of the background vorticity $\tb$ and the perturbation vorticity $\tp$. Here the colors and directions of the colored arrows indicate the amplitude and direction of the vorticity vector, with the red color corresponding to positive amplitude, and the blue color corresponding to negative amplitude. The black arrows indicate the particle trajectories of the hyperbolic flow in the $x_2-x_3$ plane generated by $v[\tb ]$.  }\label{fig1} 

We note that the two mechanism of growth, together with the structure of the background vorticity $\tb$ (which points in the direction of $e_1$) utilize all 3 spatial dimension available, see Figure~\ref{fig1} for a sketch.  Moreover, from the above discussion of the Cauchy formula \eqref{id7}, we expect that
\[
\| \otp \|_{H^\beta } \sim N^{t(1+ \beta )/2} (\lt\nt)^{\beta -s},
\]
for all $\beta \in [0,2]$, which we prove in Lemma~\ref{lem_otp}, and that the ``$1$'' inside the bracket ``$(1+\beta)$''  corresponds to the growth due to vortex stretching, while the ``$\beta$'' corresponds to the growth due to the transport term. Combining the resulting estimates of the background vorticity $\tb$ and the perturbation vorticity $\tp$, one directly obtains the norm inflation claim (Theorem~\ref{thm_black_box}).\\

In order to obtain instantaneous continuous loss  of Sobolev regularity (Theorem~\ref{thm_main}), we consider a rapidly increasing sequence $\{ D_j \}_{j\geq 1 } \subset [0,\infty )$, and we consider initial data of the form
\[
\omega_0 \coloneqq \sum_{j\geq 1} \omega_{K_j,\lambda_j } (\cdot - D_je_1,0),
\]
where $K_j\sim \varepsilon 2^{-j}$ (so that $\| \omega_0 \|_{H^s} \leq \varepsilon$), $\{ \lambda_j \}_{j\geq 1}$ is a rapidly increasing sequence with $\lambda_j\gg K_j$, and  $\omega_{K_j,\lambda_j} $ is given by Theorem~\ref{thm_black_box}. We give a gluing argument which shows that there exists $T>0$ and a classical solution $\omega(t)$ (in the sense of Definition~\ref{defn_sol}) to the Euler equation \eqref{3d_euler_vorticity} on $[0,T]$ with initial condition $\omega_0$.  We can then observe the loss of regularity by noting that, for each $\beta \in [0,2]$,
\[
\| \omega \|_{H^\beta } = \sum_{j\geq 1} \| \omega_{K_j,\lambda_j } \|_{H^\beta } + O(1) = \sum_{j\geq 1} (K_j \log \lambda_j)^{-t(1+\beta )/s} \lambda_j^{c_3(\beta-s) + t (1+\beta )c_4} + O(1),
\]
see \eqref{ooo} for details. It is thus clear that $\| \omega (t) \|_{H^\beta } < \infty $ if and only if the power of $\lambda_j$ is nonpositive, that is if $\beta \leq (s-\overline{c}t)/(1+\overline{c}t)$, where $\overline{c}\coloneqq c_4/c_3$. This in particular explains the Sobolev exponent in the claim of Theorem~\ref{thm_main}. We emphasize that the decay in time in the numerator of $(s-\overline{c}t)/(1+\overline{c}t)$ corresponds to the growth due to vortex stretching, while the decay due to the denominator corresponds to the growth due to transport, as these two numbers arise directly from the ``$1$'' and ``$\beta $'' in the exponent ``$t(1+\beta )c_4$'' above. In particular, the two effects are reflected in the loss of Sobolev regularity with the same strength, at least at $t=0$.\\

In order to prove the uniqueness part of Theorem~\ref{thm_main}, supposing that there exists another solution $\wo$, we first show that $\wo$ remains localized for $t>0$ in the sense that $\wo= \sum_{j\geq 1} \wo_j$ and each $\wo_j$ is supported in $\weta (B(\lambda_j^{-1}),t)$, where $\weta$ denotes the particle trajectory of $v[\wo ]$. We then set $W\coloneqq \omega -\wo$ we let $T_0 \in [0,T)$ be the first time when $W$ becomes nonzero. We then show that, for a short time after $T_0$, say on $[T_0,T_0+\epsilon]$, $W$ can also be localized. We then perform an error estimate on $[T_0,T_0+\epsilon ]$ in $H^2 (\R^3)$. This is challenging, since none of $\omega$, $\wo$ belong to $H^2$ and is, in fact, only possible if $D_j$ increases sufficiently fast. To be precise, we first show that $\| W\|_{H^2} \leq 1$ on $[T_0,T_0+\epsilon]$ if $D_j$ diverges sufficiently  fast, and then deduce that in fact $W=0$ on $[T_0,T_0+\epsilon]$ if $D_j$ diverges even faster, concluding the proof of uniqueness.

\subsection{Organization of the paper} The rest of the paper is organized as follows. In Section \ref{sec_prelims}, we collect notations and several useful formulas and inequalities used in later sections. Then, Theorems \ref{thm_black_box} and \ref{thm_main} are proved in Sections \ref{sec_pf_thm1} and \ref{sec_gluing}, respectively.

\section{Preliminaries}\label{sec_prelims}

\subsection{Notation}

We use the  summation convention over repeated indices. The notation ``$\supp$'' denotes the support of a function in space only. 
We use standard notation for Lebesgue spaces $L^p$ and Sobolev spaces $H^\beta$. We use the notation $\| \cdot \|_p \coloneqq \| \cdot \|_{L^p( \R^3 )}$.

Throughout the paper $c$ denotes an  absolute positive constant whose value differs from line to line.

We will sometimes use the Levi-Civita notation for the vector product
\eqnb\label{lc}
(a\times b)_i = \epsilon_{ijk} a_j b_k,
\eqne
where $i\in \{ 1,2,3\}$, $\epsilon_{ijk}$ are the Levi-Civita coefficients, and we applied the summation convention.

\subsection{Cauchy formula}
We recall the Cauchy formula,
if $v\in C([0,T]; C^1 (\R^3 ))$ is a divergence-free vector field, and $\eta (y,t)$ denotes the particle trajectories of $v$, i.e. $\p_t \eta (y,t) = v(\eta (y,t))$ with $\eta (y,0)=y$, then
\eqnb\label{cauchy_formula}
\phi (x,t) \coloneqq \nabla \eta (\eta^{-1} (x,t),t) \phi (\eta^{-1} (x,t),0)\qquad \text{is a solution to } \p_t \phi + v\cdot \nabla \phi = \phi \cdot \nabla v.
\eqne

\subsection{Gronwall estimate} We will often make use of the following Gronwall type estimate:
\eqnb\label{ode_fact1}
\text{if }f'(t) \leq c f(t) + b  \qquad \text{ then }\quad f(t) \leq f(0)\ee^{ct} + \frac{b}{c} \left( \ee^{ct}-1\right).
\eqne
In particular,
\eqnb\label{ode_fact}
\text{if } f(0)=0, \qquad \text{ then }\quad f(t)\leq bt \ee^{ct}.
\eqne
We will frequently apply this estimate in the case when $f(0)=0$ and $c=C\log N$ where $N$ is a large parameter. In that case, we obtain $f(t) \le bt N^{Ct}.$

\subsection{Fractional Sobolev space}
We recall the Sobolev-Slobodeckij characterization
\eqnb\label{ss}
\| f \|_{\dot H^s }^2 = C_s \int_{\R^3} \int_{\R^3} \frac{|f(x)-f(y)|^2}{|x-y|^{3+2s}}   \d x \,\d y \qquad \text{ for }s\in (0,1),
\eqne
see \cite[Proposition~3.4]{hitchhiker} for a proof. We note that, even if two functions $f$, $g$ have disjoint supports then $\| f+ g\|_{\dot H^s} \ne \| f\|_{\dot H^s}  +\|  g\|_{\dot H^s} $ Instead we have the following 
\eqnb\label{slobo_equiv}
\left\| \sum_{k\geq 1} f_k  \right\|^2_{\dot H^s} \sim \sum_{k\geq 1} \left\|  f_k  \right\|^2_{\dot H^s} + O\left( \sum_{k\geq 1} \| f_k \|_2^2 \right)
\eqne
for all $s\in [0,2]$, and any sequence of functions $\{ f_k\}_{k\geq 1}$ such that $\supp\, f_k \subset B(z_k,1)$, where $\{ z_k\}_{k\geq 1}$ is such that  $|z_1|\geq 2$, and $|z_k |\geq 2 |z_{k-1}|$ for all $k\geq 2$. We emphasize that the implicit constants in both $\gec $ and $\lec$ in \eqref{slobo_equiv} are independent of $s$.
\begin{proof}[Proof of \eqref{slobo_equiv}.]
We first consider $s\in (0,1)$. Let $B_k\coloneqq B(z_k,2)$, and 
\[
M \coloneqq \sum_{k\geq 1} \int_{B_k } \int_{B_k} \frac{|f_k(x)-f_k(y)|^2}{|x-y|^{3+2s } }\d x \, \d y ,
\]
which is the leading order term on each side of \eqref{slobo_equiv}. In order to compare $M$ with the first term on the right-hand side of \eqref{slobo_equiv} we note that we need to bound the integrand on the difference of the supports, that is on $\R^3 \times \R^3 \setminus \left(  B_k \times B_k \right)  = \left(  B_k^c \times \R^3 \right)  \cup \left( B_k  \times B_k^c \right) $. That is we have 
\[
\begin{split}
&\left| \sum_{k\geq 1} \left\|  f_k \right\|_{\dot H^s}^2 - M\right| =\left| \sum_{k\geq 1} \int_{\R^2} \int_{\R^2} \frac{|f_k(x)-f_k(y)|^2}{|x-y|^{3+2s } }\d x \, \d y - M \right| \\
&\hspace{1cm}\lec   \sum_{k\geq 1} \left( \int_{ B_k^c } \int_{\R^3} \frac{|f_k(x)|^2}{|x-y|^{3+2s } }\d x \, \d y+   \int_{B_k} \int_{B_k^c } \frac{|f_k(y)|^2}{|x-y|^{3+2s } }\d x \, \d y\right) \\
&\hspace{1cm}\lec   \sum_{k\geq 1} \| f_k \|_2^2 \int_{ B_k^c }  \frac{\d y }{|y-z_k|^{3+2s } } \lec   \sum_{k\geq 1} \| f_k \|_2^2,
\end{split}
\]
which gives the error term in \eqref{slobo_equiv}.

In order to estimate  $\left\| \sum_{j\geq 1} f_j \right\|_{\dot H^s}^2 - M$ we note that here the difference of the supports is  
\[
\left( \R^3 \times \R^3 \right) \setminus \left( \bigcup_l B_l \times B_l  \right) = \left( \left( \bigcup_l B_l \right)^c \times \R^3 \right)  \cup \bigcup_{l\geq 1} \left( B_l \times B_l^c \right) .
\]
 As for the first set on the right-hand side we have
\eqnb\label{te1}
\begin{split}
&\int_{\left( \bigcup_l B_l\right)^c } \int_{\R^3 } \frac{|\sum_{k\geq 1}f_k (x) -\sum_{m\geq 1 } f_m(y)|^2}{|x-y|^{3+2s } }\d x \, \d y = \sum_{k\geq 1} \int_{\left( \bigcup_l B_l\right)^c } \int_{B_k  } \frac{|f_k (x)|^2}{|x-y|^{3+2s } }\d x \, \d y \\
&\lec \sum_{k\geq 1} \| f_k \|_2^2  \int_{B_k^c  } \frac{\d y }{|y-z_k|^{3+2s } } \lec \sum_{k\geq 1} \| f_k \|_2^2 
\end{split}
\eqne
and, for the second set, we have
\[
\begin{split}
&\int_{B_l} \int_{B_l^c}\frac{|\sum_{k\geq 1}f_k (x) -\sum_{m\geq 1} f_m (y)|^2}{|x-y|^{3+2s } }\d x \, \d y = \int_{B_l} \int_{B_l^c}\frac{|\sum_{k\ne l,k\geq 1}f_k (x) - f_l (y)|^2}{|x-y|^{3+2s } }\d x \, \d y \\
&\lec \sum_{\substack{k \geq 1 \\ k\ne l}} \| f_k \|_2^2 \frac{|B_l|}{|z_k-z_l|^{3+2s}} + \| f_l \|_2^2 \lec (1+|z_l|)^{-3} \sum_{k\geq 1} \| f_k \|_2^2 + \| f_l \|_2^2,
\end{split}
\]
so that summation in $l\geq 1$ gives
\[
\sum_{l\geq 1} \int_{B_l} \int_{B_l^c}\frac{|\sum_{k\geq 1}f_k (x) -\sum_{m\geq 1} f_m (y)|^2}{|x-y|^{3+2s } }\d x \, \d y \lec \sum_{l\geq 1} \| f_l \|_2^2.
\]
This and \eqref{te1} give 
\[
\left| \left\| \sum_{k\geq 1} f_k \right\|_{\dot H^s}^2 - M \right| \lec  \sum_{k\geq 1 }  \| f_k \|_{2}^2,
\]
as required.\\

The claim is trivial for $s\in \{ 0,1,2\}$, while for $s\in (1,2)$ we apply the above analysis to $\nabla f_k$'s to obtain
\eqnb\label{temp001}
\left\| \sum_{k\geq 1} \nabla f_k  \right\|^2_{\dot H^{s-1}} \sim \sum_{k\geq 1} \left\|  \nabla f_k  \right\|^2_{\dot H^{s-1}} + O\left( \sum_{k\geq 1} \| \nabla f_k \|_2^2 \right)
\eqne
for all $s\in (1,2)$, and it remains to replace the last term by $O\left( \sum_{k\geq 1} \| f_k \|_2 \right) $. To this end, we  use Sobolev interpolation to obtain
\[
\| \nabla f_k \|_2^2 \leq \| f_k \|_2^{2\frac{s-1}s} \| f_k \|_{\dot H^s}^{\frac2s} \leq \delta \| f_k \|^2_{\dot H^s} + C_\delta \| f_k \|^2_2
\]
for each $k$ and any $\delta$, where $C_\delta>0$ does not  depend on $s$. Summing the above inequality in $k\geq 1$ we can thus absorb the $H^s$ term by the first term on the right-hand side of \eqref{temp001} to obtain the claim. 
\end{proof}

\subsection{Embeddings and nonlocal estimates}
We also recall the embeddings
\eqnb\label{Linfty_into_H^2}
\| f \|_{L^\infty} \leq L \| f \|_{H^2},
\eqne
\eqnb\label{vLinfty_into_omegaL^4}
\| v[f ] \|_{L^\infty} \leq C (\| f \|_{L^1} + \| f \|_{L^4}) \leq L \| f \|_{L^4} 
\eqne
if $| \supp\, f |\leq 1$, and
\eqnb\label{vLinfty_into_omegaL^p}
\| v[f ] \|_{L^\infty} \leq C_p (\| f \|_{L^1} + \| f \|_{L^p})  
\eqne
for any $p\in (3,\infty )$, where $L,C_p\geq 1$ are constants.

If $\omega $ is supported in $B(0,1)$ then
\eqnb\label{vel_est_general}
\| v[\omega ] \|_{C^1} \leq C_\alpha \| \omega \|_\infty  \log(2+ \| \omega \|_{C^\alpha})
\eqne
for every $\alpha \in (0,1)$. Similarly we have
\eqnb\label{Lapl_inv}
\| \p_i \p_j (-\Delta )^{-1} f \|_\infty \leq C_\alpha \| f \|_{\infty }  \log (2+ \| f \|_{C^\alpha})
\eqne
for every $\alpha \in (0,1)$, $i,j\in \{ 1,2,3 \}$.

\section{Proof of Theorem~\ref{thm_black_box} }\label{sec_pf_thm1}
Here we prove the norm inflation result, Theorem~\ref{thm_black_box}.

\subsection{The background vorticity $\tb $}

We fix $g\colon \R^3 \to [0,1]$ be such that 
$g$ is smooth,  $\supp\, g \subset B(0,1/2)$, $g=1$ on $B(0,1/4)$, \eqnb\label{temp000}
\int_{-1/2}^{1/2} \partial_{2}g(y_{1}, x_{2},x_{3}) \d y_{1}  = 0
\eqne for all $x_1, x_{2}, x_{3}$, and that $g$ is even with respect to each coordinate $x_1,x_2,x_3$. We let  
\eqnb\label{otb}
\otb (x) \coloneqq K^{-1} \begin{pmatrix}
\lambda^{\frac32-s }N^{-s} g(\lambda x) \sin (\lambda N x_2) \sin (\lambda N x_3) \\
0\\
-\lambda^{\frac52 -s} N^{-s} \sin (\lambda N x_2) \int_{1/\lambda }^{x_3} \p_1 g(\lambda x_1, \lambda x_2, \lambda y_3 ) \sin (\lambda N y_3 ) \d y_3
\end{pmatrix}.
\eqne
We let 
\eqnb\label{lambda_vs_N}
\lambda^{\frac32 -s}N^{-s} = K \log N.
\eqne
We note that $\div \, \otb =0$,  $\supp\, \otb \subset B (0,L/\lambda )$ and
\eqnb\label{otb_Ck}
\| \otb \|_{C^k } \lec_k  (\lambda N )^k \log N .
\eqne
Indeed, this is clear for $\otb_1$, and, for $\otb_3$ note that integration by parts gives
\eqnb\label{otb3_est}
\begin{split}
| \otb_3 (x) | &\lec  \lambda \log N  \left| \frac1N \int_{-1/\lambda}^{x_3} \p_1\p_3 g (\lambda x_1, \lambda x_2, \lambda y_3) \cos (\lambda N y_3 ) \d y_3 - \frac{\p_1 g (\lambda x )}{\lambda N } \cos (\lambda N x_3)   \right| \\
&\lec N^{-1} \log N ,
\end{split}
\eqne
which is, roughly speaking, $N$ times smaller than $\| \otb_1 \|_\infty$. Here we also omitted the $K$ dependence, since $K\geq 1$ and big $K$ only makes the bounds better. Similarly we obtain that 
\eqnb\label{otb3_Ck}
\| \otb_3 \|_{C^k}\lec_k   \lambda^k N^{k-1} \log N 
\eqne
for every $k\geq 0$.  Furthermore,
\eqnb\label{otb_Hb}
\| \otb \|_{H^\beta  } \sim   (\lambda N )^{\beta -s }
\eqne
for all $\beta \in [0,2]$. Indeed, a direct computation shows that $\| \nabla^k \otb \|_2 \sim_k (\lambda N)^k$ for $k=0,1,2$, and Sobolev interpolation gives \eqref{otb_Hb}.

We set 
\eqnb\label{ov_otb_def}
\ov [\otb ] \coloneqq \frac{\log N}{2 \lambda N } g(\lambda x) \begin{pmatrix}
0 \\
\sin (\lambda N x_2 ) \cos (\lambda N x_3 )\\
-\cos (\lambda N x_2 ) \sin (\lambda N x_3)
\end{pmatrix} 
\eqne
Note that 
\eqnb\label{log_defor}
\left\|   \ov [\otb ] \right\|_{C^{k,\alpha} }\lec_{k,\alpha} (\lambda N)^{k-1+\alpha }\log N 
\eqne
and
\eqnb\label{ov_otb_div}
\left\|   \div \, \ov [\otb ] \right\|_{C^k } \lec_k  \lambda^k N^{k-1} \log N  
\eqne
for $k\geq 0$, since $\div \, \ov [\otb ]$ involves differentiation of $g(\lambda \cdot )$ only.

\begin{lem}[Velocity error estimate]\label{lem_vel_err}
For every $\varepsilon\in (0,1)$, $k\geq 0$ 
\eqnb\label{vel_error}
\| v[\otb ] - \ov [\otb ] \|_{C^k} \lec_{\varepsilon ,k} \lambda^{k-1}N^{k+\varepsilon-2} \log N  
\eqne
for all sufficiently large $\lambda >1$.
\end{lem}
\begin{proof}
We set 
\[
a\coloneqq N^{-1+\frac{\varepsilon}{2}},
\]
and we let $\chi \in C_c^\infty (\R^3;[0,1])$ be such that $\chi=1$ on $B (a)$ and $\chi=0$ on $B(2a)^c$. 

For  $\otb_1 e_1$ we have that
\[
\begin{split}
u(x) &\coloneqq (v[\otb_1 e_1 ] - \ov [\otb ])(x) = \frac{-\lambda^{\frac32-s} N^{-s}}{4\pi } \int \frac{(x-y) \times ((g(\lambda y ) - g(\lambda x)) \sin (\lambda N y_2 ) \sin (\lambda N y_3 ) e_1 )}{|x-y|^3} \d y \\
&= \frac{-\lambda^{\frac12-s} N^{-s}}{4\pi } \int \frac{(\lambda x-y) \times ((g(y ) - g(\lambda x)) \sin ( N y_2 ) \sin ( N y_3 ) e_1 )}{|\lambda x-y|^3} \d y \\
&= \frac{-\lambda^{\frac12-s} N^{-s}}{4\pi } \int \frac{(\lambda x-y) \times ((g(y ) - g(\lambda x)) \chi( \lambda x - y ) \sin ( N y_2 ) \sin ( N y_3 ) e_1 )}{|\lambda x-y|^3} \d y \\
&\hspace{2cm}+ \frac{-\lambda^{\frac12-s} N^{-s}}{4\pi } \int \frac{(\lambda x-y) \times ((g(y ) - g(\lambda x)) (1-\chi (\lambda x - y) )\sin ( N y_2 ) \sin ( N y_3 ) e_1 )}{|\lambda x-y|^3} \d y \\
&=: u_{\rm in }(x) +  u_{\rm out} (x)
\end{split}
\]
As for the inner part,
\[
\begin{split}
| u_{\rm in } (x) | &\leq \frac{\lambda^{\frac12-s} N^{-s}\| g \|_{C^1} }{4\pi } \int_{\{ |y - \lambda x | \leq 2a \} } |y-\lambda x |^{-1}  \d y \leq \frac12 \lambda^{\frac12-s} N^{-s}\| g \|_{C^1} a^2.
\end{split}
\]
for every $x\in \R^3$. For the outer part we integrate by parts $m\geq 2$ times in $y_2$ to obtain
\[
\begin{split}
 u_{\rm out } (x)  &= \frac{-\lambda^{\frac12-s} N^{-m-s} }{4\pi } \sum_{|\alpha | + |\beta | + |\gamma |=m} C_{\alpha,\beta ,\gamma } \int K_{\alpha, \beta , \gamma} (x,y) \sin (N y_2 ) \sin (Ny_3) \d y ,
\end{split}
\]
where 
\[
K_{\alpha , \beta , \gamma } (x,y) \coloneqq D^\alpha \left( \frac{(\lambda x - y ) \times e_1 }{|\lambda x -y |^3 } \right) D^\beta (g(y) - g(\lambda x) ) D^\gamma (1-\chi (\lambda x - y )).
\]
Since
\[
\begin{split}
\left| \int K_{\alpha , \beta , \gamma } (x,y) \d y \right| &\leq \begin{cases}
\| g \|_\infty \int_{\{|y-\lambda x | \geq a\} } |\lambda x - y |^{-2-|\alpha |} \hspace{1cm}&|\gamma |=|\beta |=0,\\
\| g \|_{C^{|\beta |}} \int_{\{|y-\lambda x | \geq a\} \cap \supp\, g} |\lambda x - y |^{-2-|\alpha |} \hspace{1cm}&|\gamma |=0,|\beta |\geq 1,\\
\| \chi \|_{C^{|\gamma |}} \| g \|_{C^{|\beta |}} \int_{\{|y-\lambda x | \sim a\} } |\lambda x - y |^{-2-|\alpha |} \hspace{1cm}&|\gamma |\geq 1,
\end{cases}\\
&\lec_m a^{-|\alpha |  - |\gamma | +1 }\log a^{-1} \lec  a^{-m+1} \log a^{-1}
\end{split}
\]
for each $x\in \R^3$. Thus, using \eqref{lambda_vs_N}
\[
\| u \|_\infty \lec_m \lambda^{\frac12-s } ( N^{-s} a^2 + N^{-m-s} a^{-m+1} )\leq \lambda^{-1} \left( N^{-2+\varepsilon }+ N^{-\frac{m\varepsilon }2 -1+\frac{\varepsilon}2 } \right) \log N\leq \lambda^{-1} N^{-2+\varepsilon } \log N,
\]
as required, where we took $m \coloneqq \lceil2/\varepsilon \rceil$ in the last step.  The estimate for $\nabla^k u$ is analogous.

It remains to show that $\| v[\otb_3 e_3] \|_{C^k}\lec_{\varepsilon,k} \lambda^{k-1} N^{k+\varepsilon -2} \log N$. To this end we redefine
$a\coloneqq \lambda^{-1}N^{-2+\varepsilon}$ to obtain

\[
\begin{split}
| v&[\otb_3 e_3 ] (x) | =C \left| \int \frac{(x-y) \times e_3 }{|x-y|^3} \otb_3 (y )  \d y\right| \\
&\lec  \left| \int \frac{(x-y) \times e_3 }{|x-y|^3} \otb_3 (y) \chi(x-y) \d y \right| + \left| \int \frac{(x-y) \times e_3 }{|x-y|^3} \otb_3 (y) (1-\chi(x-y)) \d y \right| \\
&\lec \lambda^{\frac32 -s} N^{-1-s} \int_{B (x,2a)} \frac{ \d y }{| x-y |^{2}} + \lambda^{\frac12 -s} N^{-2-s}  \left( \int_{B(0,1) \setminus B(x,a)} \frac{\d y }{|x-y|^{3}} + \int_{\{ |x-y| \sim a \} } \frac{|\nabla \chi (x-y)|}{|x-y|^2}  \d y\right) \\
&\lec \lambda^{\frac32-s} N^{-s} (a + \lambda^{-1} N^{-2} \log a^{-1} + \lambda^{-1} N^{-2} ) \lec_{\varepsilon } \lambda^{-1} N^{-2+\varepsilon } \log N,
\end{split}
\]
where, in the second inequality, we used \eqref{otb3_Ck} and an estimate analogous to \eqref{otb3_est}, after we integrated by parts in $y_2$. This gives the required bound on $\| v[\otb_3 e_3 ] \|_\infty$, and the bounds on the $C^k$ norms follow in a similar way.
\end{proof}

Having established the error estimates for the approximation of $v[\otb]$ by the pseudovelocity $\ov [\otb ]$, we let 
\eqnb\label{def_G}
G\coloneqq \ov [\otb ] \cdot \nabla \otb - \otb \cdot \nabla \ov [\otb ]. 
\eqne
 We see that there is a cancellation in $\ov [\otb ] \cdot \nabla (\otb_1 e_1)$ when the derivatives fall on the trigonometric functions, so that
\[
\begin{split}
G(x)&= \frac{\lambda^{3-2s} N^{-1-2s}}{2}  g (\lambda x)\sin (\lambda N x_2 ) \sin (\lambda N x_3) \left(  \begin{pmatrix}
0\\ \sin (\lambda N x_2 ) \cos (\lambda N x_3 ) \\-\cos  (\lambda N x_2 ) \sin (\lambda N x_3 ) \\ 
\end{pmatrix} \cdot \nabla g (\lambda x )  e_1 \right.\\
&\hspace{4cm}\left.- \p_1 g(\lambda x ) \begin{pmatrix}
0 \\ \sin (\lambda N x_2 ) \cos (\lambda N x_3) \\
-\cos (\lambda N x_2) \sin (\lambda N x_3 )
\end{pmatrix}
\right) + \ov [\otb ] \cdot \nabla (\otb_3e_3) - \otb_3 \p_3  \ov [\otb ].
\end{split}
\]
The first term on the right-hand side is of size $\lambda^{3-2s} N^{-1-2s} = N^{-1} (\log N)^2$, which illustrates the gain of $N^{-1}$ due to the cancellation. Moreover, every derivative gives a factor of $\lambda N$. As for the last two terms on the right-hand side, $\otb_3$ is  $N$ times smaller than $\otb_1$ (recall \eqref{otb3_Ck} vs. \eqref{otb_Ck}), which gives the same gain. In other words, we have
\eqnb\label{G_error}
\begin{split}
\| G \|_{C^{k,\alpha} } &\lec (\lambda N )^{k+\alpha } N^{-1}(\log N)^2 \\
&\hspace{1cm}+ \left( \| \ov [\otb ] \|_\infty \| \otb_3 \|_{C^{k+1,\alpha }} + \| \ov [\otb ]\|_{C^{k,\alpha }} \| \otb_3 \|_{C^1} + \| \otb_3 \|_\infty \| \ov [\otb ]\|_{C^{k+1,\alpha }}+ \| \otb_3 \|_{C^{k,\alpha }} \| \ov [ \otb ] \|_{C^1}\right) \\
&\lec (\lambda N )^{k+\alpha } N^{-1}(\log N)^2
\end{split}
\eqne
for all $k\geq 0$, $\alpha \in [0,1)$. \\

We denote by $\tb$ the local-in-time strong solution to the 3D Euler equations \eqref{3d_euler_vorticity} with initial condition $\otb$. Note that, due to the assumed symmetry of $g$ we have that, for all $x,t$,
\eqnb\label{tb_symmetry}
v[\tb] (-x,t) = - v[\tb ](x,t ), \quad\text{ and, for each  }i=1,2,3,\quad  \tb_i (x,t)  \quad \text{ is odd with respect to }x_j\text{ for all }j\ne i. 
\eqne
We let  
\[
\ttb \coloneqq \tb - \otb 
\]
denote the background error. Then $\ttb$ satisfies 
\eqnb\label{ttb_eq}
\begin{split}
\p_t \ttb &= \p_t \tb = \tb \cdot \nabla v [\tb ] - v [ \tb ]\cdot \nabla \tb  + \left( - \otb \cdot \nabla \ov [\otb ] + \ov [\otb ] \cdot \nabla \otb - G \right)\\
&= \tb \cdot \nabla \left( v [ \ttb ] + (v[\otb ]-\ov [\otb ] ) \right) + \ttb \cdot \nabla \ov [ \otb ] + \left( \ov [\otb ] - v [\tb ] \right) \cdot \nabla \tb - \ov [\otb ] \cdot \nabla \ttb - G.
\end{split}
\eqne
We also set 
\eqnb\label{V_def}
V\coloneqq v[\tb ] - \ov [\otb ],
\eqne
and we note that $V$ satisfies 
\eqnb\label{V_eq}
\begin{split}
\p_t V &= - v\cdot \nabla v - \nabla p_1 = -v\cdot \nabla v + \ov \cdot \nabla \ov - \nabla p_1 -F \\
&=-V\cdot \nabla V -V\cdot \nabla \ov - \ov \cdot \nabla V  - \nabla p_1 - F,
\end{split}
\eqne
where 
\eqnb\label{what_is_F}
\begin{split}
F\coloneqq \ov \cdot \nabla \ov =& N^{-1} \log N  \,\, \ov \cdot \nabla g (\lambda x ) \begin{pmatrix}
0 \\
\sin (\lambda N x_2 ) \cos (\lambda N x_3 )\\
-\cos (\lambda N x_2 ) \sin (\lambda N x_3)
\end{pmatrix} \\
&\hspace{2cm}+ \underbrace{\lambda^{-1} N^{-1} (\log N)^2 (g(\lambda x ))^2 \begin{pmatrix}
0 \\ \sin (\lambda N x_2 ) \cos (\lambda N x_2) \\ \sin (\lambda N x_3 ) \cos (\lambda N x_3) 
\end{pmatrix}}_{=\nabla p_2 + \frac12 \lambda^{-1} N^{-2} (\log N)^2 g(\lambda x) (\cos (2 \lambda N x_2 ) + \cos (2 \lambda N x_3 )) \nabla g(\lambda x)},
\end{split}
\eqne
and where we used the fact that
\[
\begin{pmatrix}
0 \\ \sin x_2 \cos x_2 \\ \sin x_3 \cos x_3
\end{pmatrix} = -\frac14 \nabla ( \cos (2 x_2) + \cos (2 x_3) ), 
\]
and we set \begin{equation}\label{eq:p2}
	\begin{split}
		p_2 (x) \coloneqq -\frac14 \lambda^{-2} N^{-2} (\log N)^2 (g(\lambda x ))^2 (\cos (2 \lambda N x_2 ) + \cos (2 \lambda N x_3 )).
	\end{split}
\end{equation}
Thus  $\overline{F} \coloneqq F - \nabla p_2$ satisfies
\eqnb\label{est_F}
\| \overline{F}  \|_{C^{k,\alpha} } \lec \lambda^{k+\alpha -1}N^{k+\alpha -2} (\log N)^2 .
\eqne
In particular
\eqnb\label{est_F_Calpha}
\| \overline{F}  \|_{C^1 } \lec 1.
\eqne 
Moreover, note that
\eqnb\label{V_vs_vttb}
v [\ttb ] = v[\tb ] - v[\otb ] = V + \ov [\otb ] - v [\otb] .
\eqne
Note that
\eqnb\label{V_vs_vttb_norms}
\| v[ \ttb ] \|_\infty = \| V \|_\infty + O_\varepsilon (\lambda^{-1}N^{-2+\varepsilon } )
\eqne
Note that, if $\supp\, \ttb \subset B(c \lambda^{-1})$ for some constant $c>0$, then a direct estimate from the Biot-Savart law \eqref{bs} gives that $\| V \|_\infty \lec \lambda^{-1} \| \ttb \|_\infty$. This is insufficient for the control of the error $\ttb$, and instead we establish a stronger bound with an additional $N^{-1}$.
We now fix any  $\alpha \in (0, \frac12 - \frac{s}3) \subset (0,1)$, so that 
\eqnb\label{choice_alpha}
\lambda^\alpha N^{\alpha -\frac12} (\log N)^4 \qquad \text{ decreases as }\lambda \to \infty. 
\eqne

\begin{prop}[Background error]\label{prop_bckgrd_err} For any $\varepsilon \in (0,1)$, there exist absolute constants $c>0$, $T\in (0,1/2]$ such that
\eqnb\label{V_and_ttb_prop}
\begin{split}
\| \ttb \|_\infty &\lec_\varepsilon  N^{-1+\varepsilon +ct } (\log N)^2, \\
\| \ttb \|_{C^\beta} &\lec_{\varepsilon,\beta }   \lambda^{\beta}  N^{\beta-1+\varepsilon +c_\beta t } (\log N)^4 ,\\
\| V \|_\infty &\lec_\varepsilon \lambda^{-1}N^{-2+\varepsilon +ct } (\log N)^2 ,\\
\| \ttb \|_{H^{\beta'} } &\lec_{\varepsilon} N^{-1+\varepsilon + c t}(\lmb N)^{\bt'-s}(\log N)^{2}
\end{split}
\eqne
for all $\beta \in (0,1)$, $\beta' \in [0,2]$, $t\in [0,T]$ and all sufficiently large $\lambda$.
\end{prop}

\begin{rem}[A comment about the size of the errors in \eqref{V_and_ttb_prop}]
 
We note that, in light of \eqref{choice_alpha},  Proposition~\ref{prop_bckgrd_err} shows that the dynamics of the exact solution $\tb$ is determined by the dynamics of the pseudosolution $\otb$  \eqref{otb} in the $L^\infty$ and the $C^\alpha$ norm, as well as in $H^{\beta'}$ for $\beta'\in [0,2]$, provided $T_0$ is chosen sufficiently small. Indeed, $\| \otb \|_\infty \sim \log N$, $\| \otb \|_{C^{\alpha}} \sim (\lambda N)^\alpha \log N$ (by  \eqref{lambda_vs_N}, see  \eqref{otb_Ck} for the upper bounds only), and the right-hand sides in the claim \eqref{V_and_ttb_prop} of Proposition~\ref{prop_bckgrd_err} give the same quantities (respectively) multiplied by $N^{\varepsilon +ct}$ as well as a power of $\log N$. One can think of these extra factors as ``penatly factors''. Moreover, each quantity is also mutliplied by $N^{-1}$, which is the main strength of Proposition~\ref{prop_bckgrd_err} guaranteeing that the penalty factors are negligible. The factor $N^{-1}$ arises from the fact that the initial condition \eqref{otb} is a steady solution to the 3D Euler equations \eqref{3d_euler_vorticity} in the leading order frequency $\lambda N$, that is if one replaced $g$ by $1$. To be more precise, we gain $N^{-1}$ in our quantitative Riemann-Lebesgue lemma (Lemma~\ref{lem_vel_err}), as well as in the estimate \eqref{G_error} on $G$, where we observed cancellations to the leading order. We also gained a factor of $N^{-1}$ in $\otb_3$ (recall \eqref{otb3_est}).
\end{rem}
\begin{rem}[A comment about the use of the estimates \eqref{V_and_ttb_prop}]
We also note that the most important estimate in \eqref{V_and_ttb_prop} is the bound on $\| \ttb \|_\infty$, whose role is to provide a subcritical control of the error. The estimate is strong enough to also obtain a sufficiently strong estimate for $\| \ttb \|_{C^\alpha}$ for $\alpha $ small as in \eqref{choice_alpha}. We also use the $C^\alpha$ control to make the proof of Proposition~\ref{prop_bckgrd_err} simpler. 
As mentioned in the introduction (below~\eqref{id9}), in order to obtain control of $\| \ttb \|_\infty$ we need to work in both the vorticity formulation \eqref{3d_euler_vorticity} and the velocity formulation \eqref{3d_euler} of the equations, and so we also obtain the estimate on $\| V \|_\infty$.  Thanks to the first three estimates in \eqref{V_and_ttb_prop} we are able to obtain precise estimates on the particle trajectories $\eta$ of the velocity field $v[\tb ]$ (see Section~\ref{sec_part_traj} below), as well as obtain the last estimate in \eqref{V_and_ttb_prop}, which is needed for the claim regarding $\| \tb \|_{H^\beta}$ in the claim \eqref{black_box_claims} of Theorem~\ref{thm_black_box}. 
\end{rem}

\begin{proof} [Proof of Proposition~\ref{prop_bckgrd_err}.]

Let $T_0\geq 0$ be such that 
\[
\| \ttb \|_{C^\alpha } , \| V \|_\infty \leq 1
\]
for $t\in [0,T_0]$. Note that $T_0 > 0$, since $\ttb=V =0$ at time $0$ and $\otb $ (the initial condition for $\tb$) is smooth and compactly supported.

We now estimate $\| \ttb \|_\infty$, $\| \ttb \|_{C^\alpha }$, $\| V \|_\infty$ on $[0,T_0]$.

From \eqref{ttb_eq} we have
\eqnb\label{id1a}
\begin{split}
\frac{\d }{\d t} \| \ttb \|_\infty &\lec \left( \| \ttb \|_\infty + \| \otb \|_\infty \right) \left( \| v[\ttb ] \|_{C^1} + \| v[\otb ] - \ov [\otb ] \|_{C^1} \right)  +\| \ttb \|_\infty \| \ov [\otb ] \|_{C^1} \\
&+ \left( \| v[\otb ] - \ov [\otb ] \|_\infty + \| v [\ttb ]\|_\infty \right) \| \otb \|_{C^1} + \| G \|_\infty\\
&\lec_{\alpha,\varepsilon} (\| \ttb \|_\infty + \log N ) \left(  \| \ttb \|_\infty \log (2+\| \ttb \|_{C^\alpha }) + N^{-1+\varepsilon }\log N   \right) \\
&\hspace{2cm}+ (\lambda^{-1} N^{-2+\varepsilon}\log N + \| V \|_\infty ) \lambda N \log N  + N^{-1} (\log N)^2 \\
&\lec \| \ttb \|_\infty \log N + \| V \|_\infty \lambda N \log N +  N^{-1+\varepsilon} (\log N)^2,
\end{split}
\eqne
where we used \eqref{vel_est_general}, \eqref{log_defor}, \eqref{vel_error}, \eqref{otb_Ck} in the second inequality, and we noted that $\| v [\ttb ] \|_\infty \lec  \| V \|_\infty + \lambda^{-1}N^{-2+\varepsilon} \log N$ in the last line.

As for $V$, we first note that \eqref{V_vs_vttb} and \eqref{vel_error} give
\eqnb\label{V_C1}
\|  V \|_{C^1} \leq \|  v [\ttb ] \|_{C^1} + \| v[\otb ] - \ov [\otb ] \|_{C^1} \lec_{\alpha,\varepsilon}  \| \ttb \|_\infty \log (2+\| \ttb \|_{C^\alpha }) + N^{-1+\varepsilon } \log N \lec \| \ttb \|_\infty + N^{-1+\varepsilon} \log N 
\eqne
and, using boundedness of the Riesz transform in $C^\alpha $,
\eqnb\label{V_C1alpha}
\|  V \|_{C^{1,\alpha }} \leq \|  v [\ttb ] \|_{C^{1,\alpha } } + \| v[\otb ] - \ov [\otb ] \|_{C^{1,\alpha }} \lec_{\alpha ,\varepsilon} \| \ttb \|_{C^\alpha } +\lambda^\alpha  N^{\alpha -1+\varepsilon} \log N \lec 1.
\eqne

As for $\| V \|_\infty$, we get from \eqref{V_eq} that 
\eqnb\label{V_001}\begin{split}
\frac{\d }{\d t} \| V \|_{\infty} &\leq  \| V \|_{\infty} \| \ov \|_{C^{1}} + \| F \|_{\infty } + \| \nabla p \|_\infty \\
&\lec  \| V \|_{\infty} \log N + \lambda^{-1} N^{-2} (\log N)^2  + \| \nabla p \|_\infty ,
\end{split}
\eqne
where we write $\ov \coloneqq \ov [\otb ]$ for brevity. Taking the divergence of the equation  \eqref{V_eq} for $V$, we see that $p\coloneqq p_1 + p_2$ is the potential-theoretic solution of
\eqnb\label{eq_for_p_bcg}
-\Delta p = \p_i  V_j \p_j  V_i + 2\p_i  V_j \p_j \ov_i  - \div \overline{F} , 
\eqne
where we recalled from \eqref{eq:p2} that $\div \overline{F} = \div F - \Delta p_{2}$. Therefore,
\[
\nabla p = \nabla  \p_i (-\Delta )^{-1} (V_j \p_j V_i +2\p_j V_i \ov_j + 2 V_i \div \ov ) + \nabla \div (-\Delta )^{-1} \overline{F} ,
\]
which gives that, by \eqref{Lapl_inv},
\[\begin{split}
\| \nabla  p \|_{\infty } &\lec_\alpha \| V_i \p_i V \|_\infty \log (2+ \| V_j \p_j V \|_{C^\alpha } )  + \|  \ov_j  \p_j V \|_\infty \log (2+ \|  \ov_j  \p_j V \|_{C^\alpha }) \\
&\hspace{4cm}+ \| V \div\, \ov  \|_\infty \log (2+ \|  V \div \, \ov   \|_{C^\alpha } )  + \| \overline{F}  \|_\infty \log \left( 2+ \left\| \overline{F}  \right\|_{C^\alpha } \right)  \\
&\lec  \| V \|_{\infty} + \| \ov  \|_{\infty } \log(2+  \| \ov \|_{C^\alpha }) +\| V\|_{\infty } \| \div \, \ov \|_{\infty } \log (2+ \| \div \, \ov \|_{C^\alpha } ) + \lambda^{-1}N^{-2} (\log N)^2 \\
&\lec \| V \|_\infty  + \|\ttb \|_\infty \lambda^{-1} N^{-1 } \log N + \lambda^{-1}N^{-2} (\log N)^2,
\end{split}
\]
where we used \eqref{V_C1alpha} and \eqref{est_F} in the second inequality and \eqref{ov_otb_div} in the last inequality. Inserting this into \eqref{V_001} gives
\[
\frac{\d }{\d t} \| V \|_\infty \lec \| V \|_\infty \log N + \|\ttb \|_\infty \lambda^{-1} N^{-1 } \log N + \lambda^{-1}N^{-2} (\log N)^2.
\]
This and the above estimate on $\frac{\d }{\d t} \| \ttb \|_\infty$ gives that 
\[
\frac{\d }{\d t } \left( \| \ttb \|_\infty + \lambda N \| V \|_\infty \right) \lec \left( \| \ttb \|_\infty + \lambda N \| V \|_\infty \right)\log N + N^{-1} (\log N)^2
\]
Thus
\eqnb\label{V_and_ttb}
\left( \| \ttb \|_\infty + \lambda N \| V \|_\infty \right) \lec N^{-1+at } (\log N)^2 
\eqne
for $t\in [0,T_0]$, where $a>0$ is an absolute constant. In particular, H\"older interpolation and \eqref{V_C1} give 
\eqnb\label{V_Calpha}
\| V \|_{C^\alpha } \lec \| V \|_\infty^{1-\alpha} \| V\|_{C^1}^\alpha  \lec_\varepsilon \| V \|_\infty^{1-\alpha} \left( \|\ttb \|_{\infty} + N^{-1+\varepsilon} \log N \right)^\alpha  \lec \lambda^{\alpha-1}  N^{\alpha-2+\varepsilon+at } (\log N)^2 .
\eqne

This allows us to estimate $\| \ttb \|_{C^\alpha }$. We have
\eqnb\label{ttb_Calpha_comp}
\begin{split}
\frac{\d }{\d t} \| \ttb \|_{C^\alpha} &\lec \left( \| \ttb \|_{C^\alpha} + \| \otb \|_{C^\alpha} \right) \left( \| v[\ttb ] \|_{C^1} + \| v[\otb ] - \ov [\otb ] \|_{C^1} \right)+\left( \| \ttb \|_\infty + \| \otb \|_\infty \right) \left( \|\nabla  v[\ttb ] \|_{C^{\alpha}} + \| v[\otb ] - \ov [\otb ] \|_{C^{1,\alpha}} \right) \\
& +\| \ttb \|_\infty \| \ov [\otb ] \|_{C^{1,\alpha}}+\| \ttb \|_{C^\alpha } \| \ov [\otb ] \|_{C^1} \\
&+ \left( \| v[\otb ] - \ov [\otb ] \|_{C^\alpha} + \| v [\ttb ]\|_{C^\alpha } \right) \| \otb \|_{C^1}+ \left( \| v[\otb ] - \ov [\otb ] \|_\infty + \| v [\ttb ]\|_\infty \right) \| \otb \|_{C^{1,\alpha}} + \| G \|_{C^\alpha}\\
&\lec_{\alpha, \varepsilon} \left(\| \ttb \|_{C^\alpha } + (\lambda N)^\alpha \log N \right) \left(  \| \ttb \|_\infty + N^{-1+\varepsilon}\log N   \right)+\left(\| \ttb \|_\infty + \log N \right) \left(  \| \ttb \|_{C^\alpha }+ \lambda^\alpha N^{\alpha-1+\varepsilon}\log N   \right)\\
& +\| \ttb \|_\infty \lambda^\alpha N^{\alpha -1} \log N + \| \ttb \|_{C^\alpha } \log N \\
&+(\lambda^{\alpha-1} N^{\alpha -2+\varepsilon}\log N + \| V \|_{C^\alpha } ) \lambda N \log N +(\lambda^{-1} N^{-2+\varepsilon}\log N + \| V \|_\infty ) (\lambda N)^{1+\alpha} \log N  + \lambda^{\alpha } N^{\alpha -1} \log^2 N \\
&\lec \| \ttb \|_{C^\alpha } (N^{-1+\varepsilon+ct} (\log N)^2+\log N)+\| \ttb \|_\infty (\lambda N )^\alpha \log N  + \lambda^\alpha N^{\alpha -1+\varepsilon+ct } ( \log N)^3 \\
&\lec \| \ttb \|_{C^\alpha } \log N + \lambda^\alpha N^{\alpha -1+\varepsilon+ct } ( \log N)^3 
\end{split}
\eqne
Thus 
\eqnb\label{ttb_Calpha}
\| \ttb \|_{C^\alpha } \lec \lambda^\alpha N^{\alpha -1+\varepsilon+a t} (\log N)^4
\eqne
for  $t\in [0,T_0]$, where $a>0$ is an absolute constant.

Since we have obtained the estimates for $\| \ttb \|_\infty$, $\| \ttb \|_{C^\alpha}$, $\| V \|_\infty$ in \eqref{V_and_ttb_prop} on $[0,T_0]$, we now observe that, due to our choice \eqref{choice_alpha} of $\alpha$, setting $T\coloneqq 1/4a$, the estimates guarantee $\| \ttb \|_{C^\alpha }, \| V \|_\infty \leq 1$ for all $t\in [0,T]$ if $\lambda $ is taken sufficiently large. Thus, a continuity argument guarantees that the estimates \eqref{V_and_ttb_prop} also hold on $[0,T]$, as required.

The estimate for $\| \ttb \|_{C^\beta}$, for any $\beta \in (0,1)$, now follows from the same computation as in \eqref{ttb_Calpha_comp} above. \\

The proof of the $H^{\bt'}$ estimate can be done along similar lines. First, let us obtain the $L^{2}$ estimate: we rewrite the equation for $\Phi$ \eqref{ttb_eq} as \begin{equation*}
	\begin{split}
		\p_t \ttb = \tb \cdot \nabla \left( v [ \ttb ] + (v[\otb ]-\ov [\otb ] ) \right) + \ttb \cdot \nabla \ov [ \otb ] + \left( \ov [\otb ] - v [\tb ] \right) \cdot \nabla ( \Phi + \otb ) - \ov [\otb ] \cdot \nabla \ttb - G.
	\end{split}
\end{equation*} Then, we integrate it against $\Phi$ to obtain \begin{equation*}
	\begin{split}
		\frac{\ud}{\ud t} \norm{\Phi}_{2} &\lesssim \norm{ \tb \cdot \nabla v [ \ttb ] }_{2} +  \norm{ \tb \cdot \nabla (v[\otb ]-\ov [\otb ] ) }_{2} + \norm{\ttb \cdot \nabla \ov [ \otb ]}_{2} \\
		 & \qquad + \norm{( \ov [\otb ] - v [\tb ] ) \cdot \nabla \otb}_{2} +\left(\norm{\nabla \cdot ( \ov [\otb ] - v [\tb ] ) }_{\infty}+  \norm{ \div \ov [\otb ]}_{\infty}\right) \norm{\ttb}_{2} +\norm{ G}_{2} \\
		 & \lesssim \left( \norm{ \tb }_{\infty} +  \norm{\nabla \ov [ \otb ]}_{\infty} +  \norm{ \div \ov [\otb ]}_{\infty} + \norm{\nabla V}_{\infty} \right) \norm{\Phi}_{2} \\
		 & \qquad  + \norm{ \tb }_{2} \norm{\nabla (v[\otb ]-\ov [\otb ] ) }_{\infty} +  \norm{\ov [\otb ] - v [\tb ] }_{\infty}\norm{ \nabla \otb}_{2} +\norm{ G}_{2} .
	\end{split}
\end{equation*} Then, we bound \begin{equation*}
\begin{split}
	 \norm{ \tb }_{\infty} +  \norm{\nabla \ov [ \otb ]}_{\infty} +  \norm{ \div \ov [\otb ]}_{\infty}  + \norm{\nabla V}_{\infty}  \lesssim \norm{\otb}_{\infty} +  \norm{ \Phi }_{\infty} +  \norm{\nabla \ov [ \otb ]}_{\infty} +  \norm{ \div \ov [\otb ]}_{\infty}  + \norm{\nabla V}_{\infty}  \lesssim \log N 
\end{split}
\end{equation*} using \eqref{otb_Ck}, \eqref{V_and_ttb}, \eqref{log_defor},  \eqref{ov_otb_div}, and \eqref{V_C1}. Next, we bound \begin{equation*}
\begin{split}
	\norm{ \tb }_{2} \norm{\nabla (v[\otb ]-\ov [\otb ] ) }_{\infty} \le (\norm{ \otb }_{2} + \norm{ \Phi }_{2}) \norm{\nabla (v[\otb ]-\ov [\otb ] ) }_{\infty} \lesssim N^{-1+\varepsilon}\log N (\lmb N)^{-s}, 
\end{split}
\end{equation*} where we used \eqref{otb_Hb}, \eqref{vel_error}, and assumed by continuity that $\norm{ \otb }_{2} \ge \norm{ \Phi }_{2}$. Similarly, we have \begin{equation*}
\begin{split}
	 \norm{\ov [\otb ] - v [\tb ] }_{\infty}\norm{ \nabla \otb}_{2} \lesssim (\lmb N)^{-1} N^{-1+at} (\log N)^{2} (\lmb N)^{1-s}, 
\end{split}
\end{equation*} using \eqref{V_and_ttb} and \eqref{otb_Hb}. Lastly, recalling the formula of $G$ from the equation following \eqref{def_G}, we can compute explicitly that \begin{equation*}
\begin{split}
	\norm{G}_{2} \lesssim N^{-1}(\lmb N)^{-s}(\log N),
\end{split}
\end{equation*} observing that $\otb_{3}$ is $N^{-1}$ smaller in $L^2$ and in $H^{1}$ relative to $\otb$. Collecting the estimates, \begin{equation*}
\begin{split}
		\frac{\ud}{\ud t} \norm{\Phi}_{2} &\lesssim \log N  \norm{\Phi}_{2} +  N^{-1+\varepsilon+at}(\lmb N)^{-s}(\log N)^{2}
\end{split}
\end{equation*} and by Gronwall estimate, \begin{equation*}
\begin{split}
	  \norm{\Phi}_{2} \lesssim N^{-1+\varepsilon + c t}(\lmb N)^{-s}(\log N)^{2}. 
\end{split}
\end{equation*} Next, to obtain the $H^{1}$ and $H^{2}$ estimate for $\Phi$, we differentiate the equation \eqref{ttb_eq} for $\Phi$  and repeat similar estimates. We emphasize, that one of the most difficult terms in the $H^2$ estimate is the nonlinear term of the form $D^2 v[\ttb ] \cdot \nabla \ttb$. In order to estimate it in $L^2$ we cannot bound it by $\| D^2 v[\ttb ] \|_\infty \| \nabla \ttb \|_2$, since we do not even have an estimate on $\| \ttb \|_{C^1}$ (see also the remark below). Instead, we can bound it by 
\begin{equation*}
	\begin{split}
		\| D^2 v[\ttb ] \|_4 \lec \| \nabla \ttb \|_4 \lec \norm{\ttb}_{L^\infty}^{1/2} \norm{ \nabla^2 \ttb }_{L^2}^{1/2} 
	\end{split}
\end{equation*} using the Gagliardo--Nirenberg interpolation inequality in $\mathbb{R}^{3}$.

Then, the estimate for general $H^{\bt'}$ with $\bt' \in (0,2)$ follows from interpolation, concluding the proof. \end{proof}

We now fix $\varepsilon >0$ sufficiently small and take $T>0$ sufficiently small so that the claim of Proposition~\ref{prop_bckgrd_err} holds and that 
\eqnb\label{choice_vareps}
\| \ttb \|_\infty \leq N^{-1/2} , \quad \| \ttb \|_{C^\alpha } \leq (\lambda N)^\alpha N^{-1/2},\quad \| V \|_\infty \leq \lambda^{-1} N^{-3/2} 
\eqne
on $[0,T]$.

\begin{rem}\label{rmk:C1}
We note that the above control of $C^\beta$ norm of $\ttb$ does not extend directly to $C^1$. Indeed, if we want to control $\| \ttb \|_{C^1}$ then the singular integral operators (in the pressure equation as well as in the Biot Savart law $\nabla v[\ttb ] \sim SIO (\ttb )$) require us to control $\| \ttb \|_{C^{1}}$ for some $\gamma >0$. Note that, in estimating $\| \ttb \|_{C^{1,\gamma }}$ we have a term  $\| \otb \cdot \nabla \p_k v[\ttb ] \|_{\infty}$ which gives $(\log N)^2 \| \ttb \|_{C^1}$, which is too large. Indeed, estimating $\| V \|_{C^2} $ directly we also get $\| D^2 V \|_\infty (\log N )^2$ on the right-hand side (from the term $\nabla \p_j V_i \p_i \ov_j$ in the pressure estimate \eqref{eq_for_p_bcg}), which is also too much (i.e. also gives exponential growth).
\end{rem}

Note that \eqref{choice_vareps} gives in particular that
\eqnb\label{tb_linfty_Cbeta}
\| \tb \|_\infty \lec \log N
\eqne
\eqnb\label{vtb_linfty}
\begin{split}
\| v [\tb ] \|_\infty &\leq \| V \|_\infty + \| \ov [\otb ] \|_\infty \lec \lambda^{-1} N^{-2+ct } (\log N)^2 + \lambda^{\frac12-s} N^{-1-s} \lec (\lambda N)^{-1} \log N\\
\| v[\tb ] - \ov [\otb ] \|_\infty &\leq \| v[\ttb ] \|_\infty + \| v[\otb ] - \ov [\otb ] \|_\infty \leq \| V \|_\infty +2 \| v[\otb ] - \ov [\otb ] \|_\infty \leq  \lambda^{-1} N^{-2+ct } (\log N )^2 
\end{split}
\eqne
and similarly
\eqnb\label{vtb_C1}
\begin{split}
\|\nabla v [\tb ] \|_\infty &\leq \| \nabla v[\otb ] - \ov [\otb ] \|_\infty + \| \nabla \ov [\otb ] \|_\infty + \| \nabla v[\ttb ] \|_\infty \\
& \lec_\alpha N^{-1} \log N + \log N + \| \ttb \|_\infty \log (2+ \| \ttb \|_{C^\alpha } ) \\
&\lec \log N\\
\| \nabla v[\tb ] - \nabla \ov [\otb ] \|_\infty &\lec \| \nabla v[\ttb ] \|_\infty + \| v[\otb ] - \ov [\otb ] \|_{C^1} \lec N^{-1+ct } (\log N)^2
\end{split}
\eqne
for $t\in [0,T]$, where we used \eqref{vel_error}.

Moreover, since the estimate on $\| \ttb \|_{C^\alpha} $ in \eqref{V_and_ttb_prop} is in a subcritical topology, we can also obtain higher order estimates on $\tb$ on time interval $[0,T]$.
\begin{lemma}[Higher order a~priori estimates on $\phi$, $v{[}\phi {]}$]\label{lem_high}
We have that
\eqnb\label{tb_Ck}
\begin{split}
\| \tb \|_{C^{k,\beta } } &\lec_{k} (\lambda N)^{k+\beta } N^{c_{k} t} \log N ,\\
\| v[\tb] \|_{C^{k,\beta } } &\lec_{k} (\lambda N)^{k+\beta -1} N^{c_{k} t} \log N\\
\end{split}
\eqne
for all $t\in [0,T]$, $k\geq 0$, $\beta \in [0,1)$.
\end{lemma}
\begin{proof}
Indeed, the case $k=0$, $\beta = \alpha$ follows from \eqref{otb_Ck} and \eqref{V_and_ttb_prop} by writing $\| \tb \|_{C^{\alpha}} \leq \| \otb \|_{C^\alpha} + \|\ttb\|_{C^\alpha } \lec (\lambda N )^{\alpha} \log N   $, and by interpolating the first inequalities in \eqref{vtb_linfty}--\eqref{vtb_C1}. Moreover, the case $k=1$, $\beta =0$ for $v[\tb]$ follows from \eqref{vtb_C1}. For the remaining cases, we will use induction to show that, for each $k\geq 0$, $t\in [0,T]$,
\eqnb\label{ind_claim}
\| v[ \tb ] \|_{C^{k+1}} \lec_k (\lambda N)^{k} N^{c_kt } \log N,\qquad \| \tb \|_{C^{k,\alpha} } \lec_k (\lambda N )^{k+\alpha } N^{c_k t} (\log N)^2 .
\eqne
We have already verified the case $k=0$. For $k\geq 1$ we have, from the velocity formulation,
\[\frac{\d }{\d t } \| v[\tb ] \|_{C^{k+1}} \lec_k \sum_{l=0}^{k}  \| v[\tb ] \|_{C^{k+1-l}} \| v [\tb ]\|_{C^{l+1}} + \| P \|_{C^{k+2}}\]
for each $k\geq 0$,
where  $P$ is the potential-theoretic solution to 
\[
-\Delta P = \p_i v_j [\tb ] \p_j v_i [\tb ]
\]
in $\R^3$. Note that the inductive assumption gives $\| v[\tb ]\|_{C^{k,\alpha }} \lec_k (\lambda N)^{k+\alpha -1} N^{c_kt} \log N$, so that the elliptic estimate \eqref{Lapl_inv} gives
\[
\|   P \|_{C^{k+2}} \lec_k  \| v [\tb ]\otimes v [\tb ] \|_{C^k}^2 \log N.
\]
Therefore 
\[
\begin{split}
\frac{\d }{\d t } \| v[\tb ] \|_{C^{k+1}} &\lec_k \| v\|_{C^{k+1}} \log N +  \sum_{l=1}^{k-1}  \| v \|_{C^{k+1-l}} \| v \|_{C^{l+1}} + \| v [\tb ]\otimes v [\tb ] \|_{C^k}^2 \log N\\
&\lec_k \| v\|_{C^{k+1}} \log N +    (\lambda N)^{k-2} N^{c_k t} (\log N)^5 ,
\end{split}\]
so that, from the ODE fact \eqref{ode_fact1},
\[
\| v[\tb ]  \|_{C^{k+1}} \lec_k \| v[\otb ] \|_{C^{k+1}} N^{c_k t} + (\lambda N)^{k-2} N^{c_k t} (\log N)^5 \lec_k (\lambda N)^{k} N^{c_k t } \log N,
\]
where we used \eqref{vel_est_general} in the last step. In particular, $\| \tb \|_{C^1} \lec \lambda N^{1+ct} \log N$, so that the $C^{k,\alpha}$ estimate in the vorticity formulation \eqref{3d_euler_vorticity} gives   
\eqnb\label{tempp}
\begin{split}
\frac{\d }{\d t} \| \tb \|_{C^{k,\alpha}} &\lec \| v[\tb ] \|_{C^{k,\alpha }} \| \tb \|_{C^1} + \|v[\tb ] \|_{C^1} \| \tb \|_{C^{k,\alpha }} + \| \tb \|_{\infty } \| v[\tb ]\|_{C^{k+1,\alpha }}\\
& \lec_k \| \tb \|_{C^{k,\alpha} } \log N + \| \tb  \|_{C^{k-1,\alpha  }} \lambda N^{1+ct} \log N \\
&\lec_k \| \tb \|_{C^{k,\alpha} } \log N +  (\lambda N)^{k+\alpha } N^{c_kt} (\log N)^2 ,
\end{split}
\eqne
so that, again using the ODE fact \eqref{ode_fact1}, $\| \tb \|_{C^{k,\alpha} } \lec_k (\lambda N)^{k+\alpha } N^{c_kt} (\log N)^2$, completing the induction.

The claim \eqref{tb_Ck} now follows using the first estimate in \eqref{ind_claim} and H\"older  interpolation. 
\end{proof}

\subsection{The perturbation vorticity $\tp$}\label{sec_pert}
We define  
\begin{equation}\label{otp-def}
\begin{split}
	\otp (x,0) \coloneqq K^{-1} \lt^{\frac32 -s} \nt^{-s} \sin (\lt \nt x_3 )G(\lt x )  := K^{-1} \begin{pmatrix}
		\lt^{\frac52 -s }\nt^{-s} \sin (\lt \nt x_3 ) \int_{-1/\lt }^{x_1} \p_2 g (\lt h_1, \lt x_2, \lt x_3 )  \d h_1 \\
		\lt^{\frac32 -s} \nt^{-s} g(\lt x ) \sin (\lt \nt x_3 ) \\
		0
	\end{pmatrix},
\end{split}
\end{equation} where we choose 
\eqnb\label{choice_of_nt_lt}
\nt = \lt^a \qquad \text{ and } \qquad \lt= \lambda^B
\eqne
for some $a>0$, $B>1$ (fixed in \eqref{choice_a} and \eqref{choice_B} below, respectively), such that
\eqnb\label{restr1}
\lt^{\frac32-s} \nt^{-s} N^T \to 0
\eqne
as $\lambda \to \infty$, where $T>0$ is fixed by the background error estimate (recall \eqref{choice_vareps}). We note that the choice of $K\geq 1$ does not influence any of the properties of the results in this section (note that the equation \eqref{eq_otp} below is linear in $\otp$). Thus, for simplicity of the presentation we assume that 
\[
K=1
\]
for the remainder of Section~\ref{sec_pert}. 
Note that, similarly to \eqref{otb_Ck}--\eqref{otb3_Ck} we have that
\eqnb\label{otp_aniso}
\| \otp_2 (\cdot ,0) \|_2 \sim (\lt \nt)^{-s} \qquad \text{ while } \qquad \| \otp_1 (\cdot , 0) \|_2 \sim \lt^{-s} \nt^{-1-s}. 
\eqne

 We define the \emph{pseudosolution $\otp $ of the perturbation} to be the solution to
\eqnb\label{eq_otp}
\p_t \otp + v[\tb ]  \cdot \nabla \otp = \otp \cdot \nabla v[\tb ] .
\eqne
By the Cauchy formula, we can write $\otp$ explictly,
\eqnb\label{otp}
\otp (x,t) = \nabla \eta (y,t) \otp (y,0),
\eqne
where $\eta (y,t) $ the solution of the Lagrangian ODE 
\eqnb\label{new_eta_def}
\begin{split}
\p_t \eta (y,t) &= v[\tb ] (\eta(t) , t )
\end{split}
\eqne
with initial condition $\eta (y,0)=y$, and $y\coloneqq \eta^{-1} (x,t)$.

Clearly, in order to quantify the regularity of $\otp(\cdot ,t)$, we first need to better understand the particle trajectories $\eta (y,t)$, particularly for $y\in B(\lt^{-1})$. 

\subsubsection{The particle trajectories $\eta$ and estimates on the pseudosolution $\otp$}\label{sec_part_traj}
In this section we discuss some estimates on the particle trajectories $\eta (y,t)$ for $|y|\in B(\lt^{-1})$, as well as the preimage $\eta^{-1}$ (Lemma~\ref{lem_eta}) and deduce estimates on $\otp$ and $v[\otp]$ (Lemmas~\ref{lem_otp}--\ref{lem_votp}).

We first note that 
\[
\eta(0,t) =0
\]
for $t\in [0,T]$, due to the odd symmetry of $v[\tb ]$, recall \eqref{tb_symmetry}. This and \eqref{vtb_C1} show that
\[
\p_t | \eta  | \leq | v[\tb ] (\eta ) - v[\tb ] (0 ) | \leq \| \nabla v[\tb ] \|_\infty  | \eta  |\lec | \eta | \log N ,
\]
so that
\eqnb\label{where_eta}
|\eta (y,t) | \leq  N^{ct} \lt^{-1}
\eqne
for all $t\in [0,T]$, $y\in B(\lt^{-1})$. In particular
\eqnb\label{supp_otp}
\supp \, \otp  \subset B(0, \lt^{-1} N^{\widetilde{c}t} )=: B_t
\eqne
for all $t \in [0,T]$, where $\widetilde{c}>0$ is a constant, and we also used \eqref{temp000}. We need some further estimates on $\eta$ on $\supp \, \otp$, which we state in the following.
\begin{lemma}[Estimates on particle trajectories $\eta (y,t)$ for $|y|\leq \lt^{-1}$] The following estimates hold for $t\in [0,T]$, if $B>1$ is chosen sufficiently large. \label{lem_eta}
\begin{enumerate}
\item[(i)] Let $A(t) \coloneqq \nabla \eta (0,t)$. Then
\eqnb\label{nabla_eta}
A(t) =  \begin{pmatrix}
1 & 0 & 0\\
0 & N^{t/2} & 0\\
0 & 0 & N^{-t/2}
\end{pmatrix}\exp (O(N^{-1/2}))
\eqne
and
\eqnb\label{new_grad_eta}
\| \nabla \eta -A  \|_{L^\infty (B(\lt^{-1}))} \leq \lt^{-1}\lambda^{c}.
\eqne

\item[(ii)] For each $k\geq 2$,
\eqnb\label{eta_Ck}
 \begin{split}
 \| D^k \eta \|_{L^\infty (B_0 )} &\lec (\lambda N)^{k-1} N^{c_k t } (\log N)^{c_k} .
 \end{split}\eqne
\item[(iii)] Letting $y\coloneqq \eta^{-1} (x,t)$,
we have that 
\eqnb\label{nabla_y}
\| \nabla y - A^{-1} \|_{L^\infty (B_t )} \lec \lt^{-1} \lambda^c
\eqne
and, for each $k\geq 1$,
\eqnb\label{Ck_y}
\begin{split}
\| D^k  y \|_{L^\infty (B_t )} &\lec (\lambda N)^{k-1} N^{c_k t} (\log N)^{c_k}.
\end{split}
\eqne
\end{enumerate}
\end{lemma}
\begin{proof}
We first note that  \eqref{nabla_eta} follows from taking the gradient of \eqref{new_eta_def} and noting that
\[
\nabla v[\tb ] (0,t ) = \nabla \ov [\otb ](0) + O(\| \nabla v[\tb ] - \nabla \ov [\otb ] \|_\infty )= \frac{1}{2}\begin{pmatrix}
0 & 0 & 0 \\
0 & 1 & 0 \\
0 & 0 & -1
\end{pmatrix} \log N  + O\left( N^{-1+ct } ( \log N)^2\right)
\]
for $t\in [0,T]$, where we used \eqref{vtb_C1} in the last step.

As for \eqref{new_grad_eta}, we have 
\[
\p_t (\p_i \eta - Ae_i  ) =  \nabla v[\tb ] (\eta ) \cdot (\p_i \eta - Ae_i ) + \left( \nabla v[\tb ] (\eta ) - \nabla v[\tb ] (0 ) \right)  \cdot (Ae_i ) 
\]
for each $i=1,2,3$, and so, by \eqref{vtb_C1} and \eqref{tb_Ck},
\[
\begin{split}
\frac{\d }{\d t } \| \nabla \eta - A \|_{L^\infty (B_0 )} &\lec\| \nabla v[\tb ] \|_\infty \| \nabla \eta - A \|_{L^\infty (B_0 )}  + \| D^2 v[\tb ] \|_{\infty} \| \eta \|_{L^\infty (B_0 )}  |A|  \\
&\lec \| \nabla \eta - A \|_{L^\infty (B_0 )} \log N  +  \lt^{-1} \lambda^c,
\end{split}
\]
which gives \eqref{new_grad_eta}. 

For \eqref{eta_Ck}, we have, for any $i,j=1,2,3$, 
\[
\p_t \p_{ij} \eta ( y,t ) = \p_m v[\tb ] (\eta , t) \p_{ij} \eta_m +\p_{k m }  v [\tb ] (\eta ,t) \p_i \eta_m \p_j \eta_k 
\]
which gives that
\[
\begin{split}
	\frac{\d }{\d t  } \| D^2 \eta \|_{L^\infty (B_0 )} &\lec  \| \nabla v[\tb ] \|_\infty \| D^2 \eta \|_{L^\infty (B_0 )} +\| D^2 v[\tb ] \|_\infty \| \nabla \eta \|_{L^\infty (B_0 )}^2  \\
	&  \lec \| D^2 \eta \|_{L^\infty (B_0 )}  \log N + N^{ct} \lambda N (\log N )^2
\end{split}
\]
and so, since $D^2 \eta (y,0)=0$,
\[
\| D^2 \eta \|_{L^\infty (B_0 )} \lec  \lambda N^{1+ct } (\log N )^2,
\]
as required. The case of $k\geq 3$ is analogous; we only give details for $k=3$,
\[\begin{split}
\p_t \p_{ijl} \eta ( y,t ) &= \p_m v[\tb ] (\eta , t) \p_{ijl} \eta_m +\p_{mk} v[\tb ] (\eta , t) \p_l \eta_k \p_{ij} \eta_m \\
&+ \p_{k m n }  v [\tb ] (\eta ,t)\p_l \eta_n \p_i \eta_m \p_j \eta_k + \p_{k m  }  v [\tb ] (\eta ,t) \p_{il} \eta_m \p_j \eta_k + \p_{k m  }  v [\tb ] (\eta ,t) \p_{i} \eta_m \p_{jl} \eta_k 
\end{split}
\]
for all $i,j,l\in \{1,2,3\}$, so that
\[
\begin{split}
\frac{\d }{\d t} \| D^3 \eta \|_{L^\infty (B_0 )} &\lec \| \nabla v[\tb ] \|_\infty \| D^3 \eta \|_{L^\infty (B_0 )} + \| D^2 v[\tb ] \|_\infty \| D^2 \eta \|_{L^\infty (B_0 )} \| \nabla \eta \|_{L^\infty (B_0 )}  + \| D^3 v[\tb ] \|_{L^\infty (B_0 )} \|\nabla \eta \|_{L^\infty (B_0 )}^3\\
&\lec \| D^3 \eta \|_{L^\infty (B_0 )} \log N + \lambda^2 N^{2+ct} (\log N )^4,
\end{split}\]
and so \eqref{eta_Ck} follows. 

As for \eqref{nabla_y}, we have
\eqnb\label{temp02}
\nabla \eta^{-1} (\eta(y,t),t) = (\nabla \eta  )^{-1} = (A + ( \nabla \eta - A ) )^{-1} = A^{-1} \sum_{m\geq 0} \left( - A ^{-1} (\nabla \eta - A ) \right)^m
\eqne
and so
\[
| \nabla \eta^{-1} - A^{-1} | \leq \sum_{m\geq 1} | A^{-1} |^m \| \nabla \eta - A \|_\infty^m \leq  \sum_{m\geq 1} N^{2mt} \lt^{-m} \lambda^{cm}  \lec   \lt^{-1} \lambda^{c}   
\]
for every $x\in B_t$, which shows \eqref{nabla_y}as well as the case $k=1$ of \eqref{Ck_y}. 

Cases $k\geq 2$, follow by taking derivatives of \eqref{temp02}, which we now outline only for the case $k=2$. Namely, we take the gradient (with respect to $x$) of \eqref{temp02} to obtain
\[
\begin{split}
\| D^2\eta^{-1} \|_{L^\infty (B_t )} &\lec |A^{-1}| \sum_{m\geq 1} m |A^{-1}|^m \| \nabla \eta -A\|_{L^\infty (B_0 )}^{m-1} \| D^2 \eta \|_{L^\infty (B_0 )} \| \nabla \eta^{-1} \|_{L^\infty (B_t )} \\
&\lec |A^{-1}| \sum_{m\geq 1} m N^{cmt} \lt^{-(m-1)} \lambda^{c(m-1)} \lambda N^{1+ct} (\log N )^2 \\
&\lec \lambda N^{1+ct} (\log N)^2 ,
\end{split}
\]
provided $B>1$ is sufficiently large.
\end{proof}

Thanks to Lemma~\ref{lem_eta} we can obtain some precise information on the regularity of $\otp$, which we state in the following.
\begin{lem}[Estimates on $\otp$]\label{lem_otp}
	We have that  
\eqnb\label{otp_Ck}
\| \otp \|_{C^k }  \lec_k \lt^{\frac32+k-s}\nt^{k-s} N^{c_k t} 
\eqne 
 and
\eqnb\label{otp_Hbeta_norm}
\| \otp \|_{H^\beta } \sim N^{t(1+\beta )/2} (\lt \nt )^{\beta -s}
\eqne
for all  $k\geq 0$, $\beta \in [0,2]$, $t\in [0,T]$. 
\end{lem}
We note that the main difficulty of the lemma is to ensure the same constant $2^{-1}$ in the power of $N^{2^{-1}(1+\beta )t}$ in both $\lec $ and $\gec$ inequalities.

\begin{proof} 
Recalling the Cauchy formula \eqref{otp}, we see that 
\eqnb\label{temp03}
\begin{split}
\| \otp \|_{C^1} &\lec \lt^{\frac32  -s} \nt^{-s} \left(\| \nabla \eta \|_{L^\infty (B_0 )} \lt \nt \| \nabla y \|_{L^\infty (B_t )}   +\| D^2 \eta \|_{L^\infty (B_0 )}\| \nabla y \|_{L^\infty (B_t )} \right) \\
&\lec \lt^{\frac52  -s} \nt^{1-s} N^{ct}
\end{split}
\eqne
for $t\in [0,T]$, which show the case $k=1$ of \eqref{otp_Ck}. Cases $k\geq 2$ follow using  a similar computation, by taking $k$ derivatives of the Cauchy formula \eqref{otp}. As we saw in the last inequality above, the leading order term of $\| \otp \|_{C^k}$ corresponds to taking $k$ derivatives of $\otp$. Indeed, each such derivative gives a factor of $\lt\nt$, as well as a factor of $\| \nabla y \|_\infty \sim N^{t/2}$.

 For \eqref{otp_Hbeta_norm}, we first observe that 
\[
\| \otp \|_2^2 = \int \left| \nabla \eta (\eta^{-1}(x,t),t) \otp (\eta^{-1}(x,t),0) \right|^2 \d x = N^{t/2} \int |\otp_2 (y,0) |^2 \d y + L.O.T. \sim (\lt \nt )^{-s} N^{t/2},
\] 
where ``$L.O.T.$'' stands for ``lower order terms'' and where, in the second equality, we recalled that the form \eqref{nabla_eta} of $A(t)$ shows that the largest entry of $A(t)$ is $A_{22}(t) \sim N^{t/2}$. We also recalled  \eqref{new_grad_eta} that $\p_2 \eta_2 \sim A_{22}$, up to an error of the (much smaller) size $\lt^{-1}\lambda^c$, and we recalled \eqref{otp_aniso} to observe that $\otp_2 (\cdot ,0)$ is the dominant coordinate of $\otp (\cdot ,0)$. 

As for $\nabla \otp$ we note that, as in \eqref{temp03}, the leading order term of $\| \nabla \otp\|_2$ corresponds to the partial derivative $\p_3$ falling on $\otp_2$. Indeed, we have
\[
\begin{split}
\| \p_3 \otp_2 \|_2^2 &= A_{22}^2 \int | \p_3 \otp_2 (\eta^{-1}(x,t),0) \p_3 \eta^{-1}_3 (x,t) |^2 \d x + L.O.T.  \\
&= C \left( A_{22} (A^{-1})_{33} (\lt \nt )^{1-s}\right)^2+ L.O.T.\\
&\sim N^{2 t} (\lt \nt )^{2-2s},
\end{split}
\]
where we used \eqref{new_grad_eta} in the first line and \eqref{nabla_y} in the second line.

Similarly, every other partial derivative of any component of $\otp$ has the $L^2$ norm at least $\min (\nt,\lt )$ times smaller, so that $\| \nabla \otp \|_{2} \sim N^{2t} (\lt \nt )^{1-s}$ for all $t\in [0,T]$. An analogous argument for second order derivatives shows that the leading order term of $D^2 \otp$ is $ \p_{33} \otp_2 \sim A_{22} \p_3 \otp_2(\eta^{-1}(\cdot , t),0) (A^{-1})_{33}^2+L.O.T. $, and so, to be precise 
\[
\| D^2 \otp \|_{2} \sim \| \p_{33} \otp_2 \|_2 + L.O.T. \sim N^{3t/2} (\lt \nt )^{2-s}. 
\] 
Sobolev interpolation completes the proof of \eqref{otp_Hbeta_norm}.
\end{proof}
We note,  \eqref{otp_Ck} and the velocity estimate \eqref{vel_est_general} give
\eqnb\label{na_votp_Ck}
\| \nabla v[\otp ] \|_{C^k} \lec_k \| \otp \|_{C^k} \log N \lec_{k,a,B} \lt^{\frac32+k-s} \nt^{k-s} N^{c_kt} \log N
\eqne
for all $t\in [0,T]$, $k\geq 0$.
In the following we will also need some precise lowest order estimates  on $v[\otp ]$. Recall that we can easily obtain good estimates for $v[\otp ]$ in $C^k$, simply by using \eqref{vel_est_general}. However, such estimates are only valid for $k\geq 1$. 

As for $k=0$ a direct use of the Biot-Savart law \eqref{bs} only gives that 
\[
\| v [\otp ] \|_\infty \lec \lt^{-1} \| \otp \|_\infty \lec \lt^{\frac12 -s} \nt^{-s} N^t.
\]
In the following lemma we show that this can be improved by almost a factor of $\nt^{-1}$, and we also provide an $L^2$ estimate.
\begin{lem}[$L^\infty$ and $L^2$ control of  $v{[}\otp{]}$]\label{lem_votp}
If $B>1$ is sufficiently large then 
 \eqnb\label{votp_est}
	\begin{split}
		 \| v [\otp ] \|_\infty  \lec_a \lt^{\frac12 -s} \nt^{-1-s} N^{ct} \log N ,
	\end{split}
\eqne
\begin{equation}\label{votp-est-L2}
		\begin{split}
			\norm{ v[\otp] }_{L^{2}} \lesssim (\lt \nt)^{-1-s} N^{ct}\log N
		\end{split}
	\end{equation}
	for $t\in [0,T]$.
\end{lem}
\begin{proof}
For \eqref{votp_est} we let 
\[
\delta \coloneqq \lt^{-(1+a)} = \lt^{-1} \nt^{-1} \ll \lt^{-1},
\]
and we set  $\chi \in C_c^\infty (\R^3)$ be such that $\chi =1$ on $B(0,\delta)$, $\chi=0$ on $B(0,2\delta)^c$ and $|\nabla \chi |\lec \delta^{-1}$. We have
\eqnb\label{temp04}
\begin{split}
| v_i[\otp ] (z) | &= \left| \frac{1}{4\pi } \int_{\R^3}  \epsilon_{ijk} \frac{z_j-x_j}{|z-x|^3}   \otp_k (x)  \d x  \right| \\
 &= \left| \frac{1}{4\pi } \int_{\R^3}  \epsilon_{ijk} \frac{z_j-x_j}{|z-x|^3}   \p_l \eta_k (y(x) )\otp_l (y(x),0)  \d x  \right| \\
 &= \left| \frac{1}{4\pi } \int_{\R^3}  \epsilon_{ijk} \frac{z_j-\eta_j(y)}{|z-\eta(y)|^3}   \p_l \eta_k (y )\otp_l (y,0)  \d y  \right| \\
 &\lec \| \otp (\cdot , 0) \|_\infty \int_{\R^3}   \frac{\nabla \eta (y) \chi (z-\eta(y) )}{|z-\eta(y)|^2}    \d y  +\left| \int_{\R^3}  \epsilon_{ijk} \frac{z_j-\eta_j(y)}{|z-\eta(y)|^3}   \p_l \eta_k (y )\otp_l (y,0) (1-\chi (z-\eta(y) )) \d y   \right| \\
 &\lec  \lt^{\frac32-s }\nt^{-s}N^{ct}  \int_{B (0,2\delta ) } | x|^{-2} \d x \\
 &+ \lt^{\frac12-s} \nt^{-1-s} \int_{ B_t } \left( \left( \frac{|\nabla \eta (y) |^2}{|z-\eta(y) |^{3}}  + \frac{|D^2 \eta (y) |}{|z-\eta (y)|^{2}}  +\frac{|\nabla \eta (y)| \lt  \|\nabla g \|_\infty}{|z-\eta(y) |^{2}}  \right) (1-\chi(z-\eta (y))) \right. \\
 &\hspace{9cm}\left. + \frac{|\nabla \eta(y) |^2 |\nabla \chi (z- \eta(y)  )|}{|z-\eta (y) |^{2}}  \right) \d y\\
 &\lec \lt^{\frac32-s}\nt^{-s}  \left( N^{ct} \delta + \lt^{-1}\nt^{-1}N^{ct} \log \delta^{-1}+\lt^{-2}\nt^{-1}\lambda N^{1+ct} \log N  + \lt^{-1}\nt^{-1} N^{ct}+ \lt^{-1} \nt^{-1} N^{ct} \right)\\
 &\lec_a \lt^{\frac12 -s }\nt^{-1-s} N^{ct} \log N
 \end{split}
\eqne
for each $i\in \{ 1,2,3\}$ and every $z\in \R^3$, where we used the fact that $\mathrm{det}\, \nabla \eta =0$, due to incompressibility, and, in the second inequality, we integrated by parts in $y_3$ in the term involving $(1-\chi)$ (by integrating $\otp_l$).  We also used \eqref{eta_Ck} in the third inequality. (Recall also \eqref{supp_otp} that $B_t = B(\lt^{-1} N^{\widetilde{c}t})$.) This shows \eqref{votp_est}.

As for \eqref{votp-est-L2}, the local part can be estimated using \eqref{votp_est} directly,
\[
\| v[\otp ] \|_{L^2 (B(2 N^{\widetilde{c}t} \lt^{-1} ))} \lec \lt^{-\frac32 }N^{ct} \| v[\otp ] \|_{L^\infty } \lec (\lt \nt )^{-1 -s } N^{ct} \log N.
\]
Moreover, for $z\in B(2 N^{\widetilde{c}t} \lt^{-1} )^c$, we can use the first three equalities of \eqref{temp04} and integrate by parts in $y_3$ to obtain
\[
\begin{split}
| v_i[\otp ] (z) |  &=c \lt^{\frac12-s} \nt^{-1-s}\left|\epsilon_{ijk} \int \p_{y_3} \left( \frac{z_j-\eta_j(y)}{|z-\eta(y)|^3}   \p_l \eta_k (y ) G (\lt y ) \right)  \cos (\lt \nt y_3 ) \d y \right|\\
&\lec \lt^{\frac12-s} \nt^{-1-s} \int_{B_t } \left( |z- \eta (y) |^{-3} |\nabla \eta (y) |^2 +|z- \eta (y) |^{-2} |D^2 \eta (y) | +|z- \eta (y) |^{-2} |\nabla \eta (y) | \lt \right) \d y \\
&\lec \lt^{\frac12-s} \nt^{-1-s} \left(\lt^{-3}N^{ct} |z|^{-3} + \lt^{-2}  |z|^{-2} \right)  \\
&\lec \lt^{-\frac32 -s }\nt^{-1-s} N^{ct}  |z|^{-2}
 \end{split}
\]
for each $i=1,2,3$, so that
\[
\| v[\otp ] \|_{L^2(B(2 N^{\widetilde{c}t} \lt^{-1} )^c)}^2  \lec \lt^{-3 -2s }\nt^{-2-2s} N^{ct} \int_{\{ |z| \geq 2 N^{\widetilde{c}t} \lt^{-1} \} } |z|^{-4} \d z \lec \lt^{-2 -2s }\nt^{-2-2s} N^{ct} ,
\]
as required.
\end{proof}

\subsubsection{Control of the perturbation error}
We now define $\tp$ to be such that $\tp+\tb$ solves the $3$D Euler equation \eqref{3d_euler_vorticity} with initial data $\otp(x,0) + \otb (x)$, i.e.
\[
\p_t (\tp +\tb ) +v[\tp + \tb ] \cdot \nabla (\tp +\tb ) = (\tp +\tb ) \cdot \nabla v[\tp + \tb ].
\]
We define  
\[
\ttp \coloneqq \tp - \otp, \qquad  W \coloneqq v[\ttp ] =v[\tp ] - v [\otp ] 
\]
as the perturbation error and the velocity error for the perturbation, respectively. We now fix $a >1$ to be any number such that  
\eqnb\label{choice_a}
\frac32 -(1+a) s <0,
\eqne
namely that $a>(3-2s)/s$. In particular, note that
\[
\nt \geq \lt.
\]
In the error estimates below we will often use the prefactor of 
\[
\lt^{\frac32 -s } \nt^{-s},
\]
which is the leading order size of $\| \otp \|_\infty$, due to \eqref{otp_Ck}. In particular \eqref{choice_a} means that $\| \otp\|_\infty$ can be bounded by some negative power of $\lt$, and so, if $B>1$ is large enough, this gives us a smallness that overcomes any growth we may accumulate from any power of $\lambda$ (arising from the terms involving the background vorticity $\tb$). We also emphasize that this aspect of the error analysis of the perturbation vorticity $\tp$ differs from the analysis of the background vorticity (Proposition~\ref{prop_bckgrd_err}), in which case every term  involving  $\| \tb \|_\infty \sim \log N$ had to be handled more carefully. \\

 We also fix $\gamma \in (0,1)$ sufficiently small so that
\eqnb\label{choice_of_gamma}
\frac32 +(1+a)(\gamma - s ) <0\qquad \text{ and } \qquad (2+a)\gamma <1  <0 ,
\eqne
so that, apart from being able to estimate $\| \tp \|_\infty$, we also control $\| \tp \|_{C^\gamma}$ (while the second property in \eqref{choice_of_gamma} in used in Step 3 below. In particular, recalling \eqref{otp_Ck} again, we have that 
\eqnb\label{otp_Cgamma_shrinks}
\| \otp \|_{C^\gamma }  \log N \to 0 
\eqne
as $\lambda \to \infty$,  uniformly on $[0,T]$, if $B>1$ is sufficiently large. We now restrict our attention  to the choices of $B>1$ sufficiently large so that there exists   $c_*>0$ such that
\eqnb\label{choice_cstar}
\lambda N \lt^{-1} + \lt^{\frac32 -s} \nt^{-s} \leq  \lambda^{-c_*}  .
\eqne

\begin{prop}[Perturbation error]\label{prop_pert_error}
There exists  $\overline{c} >0$, such that
\eqnb\label{W_and_ttp_prop}
\begin{split}
\| \ttp \|_\infty + (\lt \nt)\| W \|_\infty  &\lec  \lt^{\frac32-s} \nt^{-s}  \lmb^{-c_*}  N^{ \overline{c} t} (\log N)^{3} , \\
  (\lt \nt)^{-\gmm}\| \ttp \|_{C^\gamma} &\lec \lt^{\frac32-s} \nt^{-s}  \lmb^{-c_{*}}  N^{\overline{c} t} (\log N)^{3},\\
\| \ttp \|_{H^{\beta} } &\lec  (\lt\nt)^{\bt-s} \lmb^{ -c_* }   N^{\overline{c}t} (\log N)^2
\end{split}
\eqne
for all $t\in [0,T]$, $\beta \in [0,s]$ and all sufficiently large $\lambda$, if  $B>1 $ is sufficiently large. 
\end{prop}
\begin{rem}[The estimates in~\eqref{W_and_ttp_prop}]
We note that the main point of the proposition above is that  $c_* >0$  (which is possible only if  $B>1$ is sufficiently large), so that the estimates \eqref{W_and_ttp_prop} are negligible as compared to the $L^\infty$, $C^\gamma$, $H^\beta$ estimates \eqref{otp_Ck}--\eqref{otp_Hbeta_norm} on $ \otp$. Similarly to Proposition~\ref{prop_bckgrd_err} on the background vorticity error estimates, the first two estimates in \eqref{W_and_ttp_prop} let us control $\| \ttp \|_{H^\beta}$, which is needed for the $H^\beta$ part of the claim \eqref{black_box_claims} of Theorem~\ref{thm_black_box}.
\end{rem}
\begin{proof}[Proof of Proposition~\ref{prop_pert_error}.]
We first let $T_0\in (0,T] $ (where $T$ is given by background vorticity error, recall~\eqref{choice_vareps}) be the largest time such that  $\| \ttp \|_{C^\gamma}, \| W \|_\infty \leq 1$ on $[0,T_0]$. We perform estimates for $t\leq T_0$.

\medskip 
\noindent \textbf{Step 1. Estimate of $\| \ttp \|_\infty$.} We see that
\eqnb\label{eq_ttp}
	\begin{split} 
		\p_t \ttp &= - v[\tb ] \cdot \nabla \ttp - v[\tp ] \cdot \nabla (\tp +\tb ) + \ttp \cdot \nabla v[\tb ]   + (\tp +\tb )\cdot \nabla v[\tp ] \\
		&= 
		- (v[\tb ] + v[\psi] )\cdot \nabla \ttp 
		- v[\tp]\cdot( \nabla \otp + \nabla \tb  ) 
		+ \ttp \cdot \nabla v[\tb ] 
		+ (\tp +\tb )\cdot \nabla v[\tp ]
\end{split}
\eqne
so that moving over the transport term, we obtain \begin{equation}\label{ttp_eq}
	\begin{split}
		\left( \partial_t + (v[\phi] + v[\psi]) \cdot \nabla \right) \ttp = I + II + III ,
	\end{split}
\end{equation} where
 \begin{equation*}
	\begin{split}
		I \coloneqq -v[\tp]\cdot( \nabla \otp + \nabla \tb  ) , \quad II \coloneqq  \ttp \cdot \nabla v[\tb ] , \quad  III = (\tp +\tb )\cdot \nabla v[\tp ]
	\end{split}
\end{equation*} and further rewrite \begin{equation*}
	\begin{split}
		III &\coloneqq \tp\cdot \nabla v[\tp ] + \tb\cdot \nabla v[\tp ]  = \tp \cdot \nabla v[ \otp ] + \tp \cdot \nabla v[ \ttp] +  \tb \cdot \nabla v[ \otp ] + \tb \cdot \nabla v[ \ttp] \\
		&=: III_1 + III_2 + III_3 + III_4. 
	\end{split}
\end{equation*} 
In order to estimate $\|\ttp \|_\infty$ we first note that one of the main term in that estimate will be
\eqnb\label{III3}
\| III_{3} \|_\infty = \| \tb \cdot  \nabla v[\otp ] \|_\infty \leq  \lt^{\frac32-s} \nt^{-s} N^{ct}\log N 
\eqne
where we used \eqref{tb_linfty_Cbeta}, as well as the velocity estimate \eqref{vel_est_general} and we recalled \eqref{otp_Cgamma_shrinks} that $\| \otp \|_{C^\gamma } \lec 1$. However, the last upper bound is too na\"ive estimate, and in fact insufficient, as it is larger than our desired bound in \eqref{W_and_ttp_prop}. 

We now make it almost $\lt^{-1}$ sharper by showing the following. \\

\textbf{Claim:}	For all $t\in [0,T]$,
	\eqnb\label{tb_na_v_otp}
	\| \tb \cdot \nabla v[\otp ] \|_\infty \lec \lt^{\frac32-s} \nt^{-s} N^{ct} (\log N)^2 \left( \lt^{-1 } \lambda N +\nt^{-1}  \right) 
	\eqne
\begin{proof}[Proof of Claim.]
	Indeed, recalling the symmetry \eqref{tb_symmetry} we note that  $\tb =0$ for $x=0$, and so \eqref{tb_Ck} gives
	\eqnb
	\begin{split}
		|\tb (x) | &\lec \|  \tb \|_{C^1} |x| \lec \lambda N^{1+ct} \log N |x| .
	\end{split}
	\eqne
	Thus
	\[
	\| \tb \cdot \nabla v[\otp ] \|_{L^\infty (B(2\lt^{-1} N^{\widetilde{c}t}))} \lec \lambda N^{1+ct}  \log N (\lt^{-1} N^{\widetilde{c}t }) \, \| \nabla v[\otp ]\|_\infty \lec_{a} \lt^{\frac12-s}\nt^{-s} \lambda N^{1+ct}  (\log N)^2.
	\]
	As for $|x| \geq  2\lt^{-1}N^{2t}$ we see that, similarly to the proof of Lemma~\ref{lem_votp}, 
	\eqnb\label{votp_away}
	\begin{split}
		|\p_m v_i [\otp ] (z) | &=C \left| \int \p_m  \frac{(z-x)}{|z-x|^3} \times  \otp_i (x)\d x\right|=C  \left|\epsilon_{ijk} \int \p_m \frac{ (z_j-x_j)  }{|z-x|^3} \p_l \eta_k (y(x) ) \otp_l (y(x),0) \d x\right|\\
		&=C  \left|\epsilon_{ijk} \int \p_{\eta_m} \frac{ (z_j-\eta(y)_j)  }{|z-\eta(y)|^3} \p_l \eta_k (y ) \otp_l (y,0) \d y\right|\\
		&\lec \lt^{\frac12-s}\nt^{-1-s} \int_{B_t } \left( \frac{|\nabla \eta (y) |^2}{|z-\eta(y)|^4} +\frac{|D^2 \eta (y) |}{|z-\eta(y)|^3} + \frac{\lt |\nabla \eta (y) |}{|z-\eta(y)|^3}  \right) \d y\\
		&\lec \lt^{\frac12-s}\nt^{-1-s} \left( N^{ct}(\lt^{-1}N^{\widetilde{c}t})^{-4} + \lambda N^{1+ct } \log N  (\lt^{-1}N^{\widetilde{c}t})^{-3} + \lt N^{ct} (\lt^{-1}N^{\widetilde{c}t})^{-3} \right) |B_t |\\
		&\lec \lt^{\frac32-s}\nt^{-1-s} \lambda N^{1+ct} \log N 
	\end{split}
	\eqne
	for $m,i=1,2,3$, where we integrated by parts in $y_3$, and used \eqref{new_grad_eta}, \eqref{eta_Ck} in the fourth line. Recalling that $\| \tb \|_\infty \lec \log N$ completes the proof of Claim.
\end{proof}

Having verified \eqref{tb_na_v_otp}, we can now estimate $\| \ttp \|_\infty$ using \eqref{ttp_eq}. To begin with, using Lemma \ref{lem_votp} \begin{equation*}
\begin{split}
		\| I \|_{\infty} & \lesssim \| v[\tp ] \|_\infty \| \otp + \tb \|_{C^1} \lesssim \left( \| v[\ttp ] \|_\infty +\| v[\otp ] \|_\infty \right) ( (\lt \nt) \| \otp \|_\infty  + \lambda N \log N ) \\
		& \lesssim ( \| W \|_\infty +(\lt\nt)^{-1} \log N \| \otp \|_\infty)((\lt \nt) \| \otp \|_\infty + \lambda N  \log N ) \\
		& = (  (\lt \nt)\| W \|_\infty + \log N \| \otp \|_\infty)( \| \otp \|_\infty + (\lt \nt)^{-1}\lambda N  \log N ).
\end{split}
\end{equation*}
The  term $II$ is bounded by 
\begin{equation*}
	\begin{split}
		\| II \|_{\infty} \lesssim \log N \| \ttp \|_\infty .
	\end{split}
\end{equation*} Lastly, using \eqref{votp_est}, 
\begin{equation*}
\begin{split}
	\| III \|_{\infty} & \lec  \| \tp \|_\infty \| \nabla v[\otp ] + \nabla v[\ttp ] \|_\infty + \| \tb \cdot  \nabla v[\otp ] \|_\infty + \| \tb \|_\infty \|  \nabla v[\ttp ] \|_\infty \\
	& \lec (\| \ttp \|_\infty + \| \otp \|_\infty )   \left(  \| \otp \|_\infty \log (2+ \| \otp \|_{C^\gamma }) + \| \ttp \|_\infty \log (2 +\| \ttp \|_{C^\gamma}) \right) \\
	& \qquad  +  \| \otp \|_\infty N^{ct} (\log N )^2 \left( \lt^{-1} \lambda N + \nt^{-1} \right)  + \| \ttp\|_\infty (\log N )\log (2+ \| \ttp \|_{C^\gamma })
	\\
	& \lec (\| \ttp \|_\infty + \| \otp \|_\infty )^{2}  +  \| \otp \|_\infty N^{ct} (\log N )^2 \left( \lt^{-1} \lambda N + \nt^{-1} \right)  + \| \ttp\|_\infty (\log N ) . 
\end{split}
\end{equation*} Combining the terms, we obtain that \begin{equation*}
		\begin{split}
			\p_t \| \ttp \|_\infty &\lec \left(\log N +  \| \otp \|_\infty + \| \ttp \|_\infty  \right) \| \ttp \|_\infty  +   \left( \| \otp \|_\infty + (\lt \nt)^{-1}\lambda N  \log N \right)  (\lt \nt)\| W \|_\infty\\
			&\qquad    +  \left( (1 + \log N) \| \otp \|_\infty + (\lt \nt)^{-1} \lambda N (\log N)^2  +  N^{ct} (\log N )^2 \left( \lt^{-1} \lambda N + \nt^{-1} \right)   \right)  \| \otp \|_\infty. 
		\end{split}
\end{equation*}
Using that $ \| \otp \|_\infty, \| \ttp \|_\infty \ll 1 $, $ (\lt \nt)^{-1} \lambda N \ll 1$, this simplifies to \begin{equation}\label{ttp_est}
		\begin{split}
			\p_t \| \ttp \|_\infty &\lec  \log N  \left( \| \ttp \|_\infty  +     (\lt \nt)\| W \|_\infty \right)   +  \left(  \| \otp \|_\infty + (\lt \nt)^{-1} \lambda N   +  N^{ct}  \left( \lt^{-1} \lambda N + \nt^{-1} \right)   \right) (\log N)^{2}  \| \otp \|_\infty \\
			&\lec  \log N  \left( \| \ttp \|_\infty  +     (\lt \nt)\| W \|_\infty \right)   +  \lmb^{-c_*} N^{ct} (\log N)^{2}  \| \otp \|_\infty.
		\end{split}
	\end{equation}

\noindent \textbf{Step 2. Estimate of $\| W \|_\infty$.}  In order to derive the PDE for $W$ we first use \eqref{eq_otp} to obtain that
\eqnb\label{votp_eq}
\p_t v[\otp ] + v[\tb ] \cdot \nabla v[\otp] + v[\otp ] \cdot \nabla v[\tb ] + \nabla p =0 ,
\eqne
where $p$ is the potential-theoretic solution of 
\[
-\Delta p = 2 \p_j  v_i[\phi]  \p_i v_j [\otp ] 
\]
in $\R^3 $, for each $t$. Indeed, recalling the identity
\eqnb\label{eq_calculus}
\nabla \times (a\cdot \nabla b) = a\cdot \nabla (\nabla \times b) - (\nabla \times a ) \cdot \nabla b
\eqne
we have that, for $U\coloneqq v[\tb ] \cdot \nabla v[\otp] + v[\otp ] \cdot \nabla v[\tb ] + \nabla p$, $\mathrm{div}\, U=0$, and
\[
\curl \,U  = v[\tb ] \cdot \nabla \otp - \otp \cdot \nabla v[\tb ] ,
\]
which shows that $v[v[\tb ] \cdot \nabla \otp - \otp \cdot \nabla v[\tb ] ]=U$, and hence taking $v[\cdot ]$ of \eqref{eq_otp} we obtain \eqref{votp_eq}. Moreover, since $\tb + \tp $ solves the Euler equation \eqref{3d_euler_vorticity}, 
\eqnb\label{vtp_eq}
\p_t v [\tp ] 
= \p_t v[\tp +\tb] - \p_t v[\tb ] 
= -v[\tp + \tb ] \cdot \nabla v [\tp  ] -v[\tp ] \cdot \nabla v[\tb ] + \nabla q,
\eqne
where $q$ is the potential-theoretic solution of 
\eqnb\label{q_eq}
\Delta q =  \p_i \p_j (  v_i [\tp + \tb ] v_j [\tp ]  + v_i [\tb ] v_j [\tp ] )=  \p_i ( 2 v[\tp ] \cdot \nabla  v_i [ \tb ]  +v[\tp ] \cdot \nabla  v_i [\tp ] ).
\eqne
Thus, combining \eqref{votp_eq} and \eqref{vtp_eq}, we see that \begin{equation*}
	\begin{split}
		\p_t W &= \p_t v[\tp ] - \p_t v[\otp ] = -v[\tp + \tb ] \cdot \nabla v [\tp  ] -v[\tp ] \cdot \nabla v[\tb ] + \nabla q + v[\tb ] \cdot \nabla v[\otp] + v[\otp ] \cdot \nabla v[\tb ] + \nabla p,
	\end{split}
\end{equation*} which gives after rearranging
\eqnb\label{eq_for_W}
\begin{split}
(\p_t + v[\tp + \tb ] \cdot \nabla )W =  - (W +v[\otp ]  )\cdot \nabla v[\otp ] - W \cdot \nabla v[\tb ]  
 + \nabla Q 
\end{split}
\eqne
where $Q = q+p$ is the potential-theoretic solution of 
\eqnb\label{Q_eq}
- \Delta Q = - 2 \p_j W_i ( \p_i v_j [\tb ] + \p_i v[\otp ] ) - \p_j W_i \p_i W_j - \p_j v_i [\otp ] \p_i v_j [\otp ] .
\eqne
We have that
\eqnb\label{vel_W_est}
\begin{split}
\frac{\d }{\d t} \| W \|_\infty &\lec \| W \|_\infty \left( \|\nabla v[ \otp ] \|_\infty  + \| \nabla v [\tb ] \|_\infty \right) + \|v[\otp ] \cdot  \nabla v[\otp ] \|_\infty    + \| \nabla Q \|_\infty  \\
& =: \| W \|_\infty \underbrace{ \left( \|\nabla v[ \otp ] \|_\infty  + \| \nabla v [\tb ] \|_\infty \right)}_{\lec \log N }  +  \|IV \|_\infty    + \|\nabla Q  \|_\infty 
\end{split}
\eqne
where we recalled \eqref{vtb_C1} and used $$\| \nabla v[\otp ] \|_\infty \lec \| \otp \|_\infty \log (2+ \| \otp \|_{C^\gamma } ) \lec \lt^{\frac32 -s } \nt^{-s} N^{ct} \lec 1$$ to bound the expressions in the bracket by $\log N$. For $IV = v[\otp ] \cdot  \nabla v[\otp ]$, we use Lemma \ref{lem_votp} and the previous estimate $\| \nabla v[\otp ] \|_\infty \lec \| \otp \|_\infty \log (2+ \| \otp \|_{C^\gamma } )$ to get  
\begin{equation*}
	\begin{split}
		\norm{ IV}_\infty \lesssim (\lt \nt)^{-1} \| \otp \|_\infty^{2} N^{ct} \log N  \lesssim (\lt \nt)^{-1} \| \otp \|_\infty \lmb^{ (\frac32 - (1+a)s)B } N^{ct} \log N.  
	\end{split}
\end{equation*}

For the pressure term, we have from \eqref{Q_eq} and $\| \otp  \|_\infty \log (2+ \| \otp \|_{C^\gamma} ) + \| \ttp \|_\infty \log(2 + \| \ttp \|_{C^\gamma} ) \lesssim 1$ that
\eqnb\label{q_est}
\begin{split}
\| \nabla Q \|_\infty &\lec   \| W \|_\infty \left( \| \nabla v[\tb ] \|_\infty + \| \nabla v[\otp ] \|_\infty \right) + \| W \|_\infty \| \nabla W \|_\infty + \| v[\otp ] \cdot \nabla v[\otp ] \|_\infty   \\
&\lec \| W \|_\infty \left( \log N + \| \otp  \|_\infty \log (2+ \| \otp \|_{C^\gamma} ) + \| \ttp \|_\infty \log(2 + \| \ttp \|_{C^\gamma} ) \right) + 	\norm{ IV}_\infty \\
&\lec  \| W \|_\infty  \log N   + 	\norm{ IV}_\infty,
\end{split}
\eqne
as required, where, in the first line, we used the fact that the logarithmic term is $O(1)$.  Indeed, 
\[
\begin{split}
&\| 2W \cdot \nabla (v[\tb ] +  v[\otp ] ) +  W \cdot \nabla W  +  v[\otp ] \cdot \nabla v[\otp ] \|_{C^\gamma} \\
& \quad \lec_\gamma   \| W \|_{C^\gamma} \left( \| \nabla v [\tb ] \|_\infty + \| v [\otp ] \|_\infty \right) + \| W \|_\infty \left( \| \tb  \|_{C^\gamma} + \| \otp  \|_{C^\gamma} \right)  +  \| W \|_{C^\gamma } \| \nabla W \|_\infty  +\| W \|_\infty \| \ttp \|_{C^\gamma} + \| \otp  \|_{C^\gamma}^2\\
&\quad \lec_\gamma   \| \ttp \|_{C^\gamma } (\log N + \lt^{\frac12-s} \nt^{-1-s} N^{ct} \log N+ \| \ttp \|_\infty )+ 1 \lec_{\gamma, B} 1
\end{split}
\]
where we used the fact that $\| W \|_{C^\gamma} \lec \| \ttp \|_{C^\gamma} $, \eqref{votp_est} in the third inequality.

Collecting the terms, the velocity estimate \eqref{vel_W_est} becomes
\[
\frac{\d }{\d t } \| W \|_\infty \lec  \| W \|_\infty   \log N   + (\lt \nt)^{-1} \| \otp \|_\infty \lt^{\frac32-s} \nt^{-s} N^{ct} \log N
\] and combining this with \eqref{ttp_est} and $ \| \otp \|_\infty \lec \lt^{\frac32 -s } \nt^{-s} N^{ct}$ (recall \eqref{otp_Ck}) we obtain \begin{equation*}
	\begin{split}
		\frac{\d }{\d t } \left( \| \ttp \|_\infty + \lt \nt \| W \|_\infty \right) \lec \left( \| \ttp \|_\infty + \lt \nt \| W \|_\infty \right)    \log N   +  \lt^{\frac32 -s } \nt^{-s}  \lmb^{-c_*}  N^{ c t} (\log N)^{2}   .
	\end{split}
\end{equation*}
  This gives \begin{equation}\label{ttp_and_W_est}
\begin{split}
	\| \ttp \|_\infty + \lt \nt \| W \|_\infty  \lec  \lt^{\frac32 -s } \nt^{-s}  \lmb^{-c_*}   N^{ c t} (\log N)^{3} 
\end{split}
\end{equation} for $t\in [0,T]$.

\medskip

\noindent \textbf{Step 3. Estimate of $\| \ttp \|_{C^\gamma}$.} We return to \eqref{ttp_eq} and write the equation along the flow $X(x,t)$ generated by the velocity $v[\phi]+v[\psi]$: for any two points $x \ne x'$, we compute \begin{equation*}
	\begin{split}
		&\frac{\d}{\d t}\left( \frac{\ttp(X(x)) - \ttp(X(x'))}{ |X(x) - X(x')|^{\gamma} } \right)  = \frac{H(X(x)) - H(X(x'))}{|X(x) - X(x')|^\gamma} \\
		&\qquad  - \gamma (\ttp(X(x)) - \ttp(X(x'))) \frac{ ( X(x) - X(x') ) \cdot ( (v[\phi]+v[\psi])(X(x)) - (v[\phi]+v[\psi])(X(x'))  )  }{|X(x) - X(x')|^{\gamma+2}},
	\end{split}
\end{equation*} where we omitted ``$t$'' in the notation and
\[
H \coloneqq I + II + III
\]
denotes the right hand side of \eqref{ttp_eq}. Taking absolute values and using the mean value theorem for $v[\phi]+v[\psi]$,  \begin{equation}\label{eq:C-gmm}
\begin{split}
	\frac{\d}{\d t}  \frac{|\ttp(X(x)) - \ttp(X(x')) |}{ |X(x) - X(x')|^{\gamma} }  \le \| H \|_{C^\gamma} + \gamma  \frac{|\ttp(X(x)) - \ttp(X(x')) |}{ |X(x) - X(x')|^{\gamma} } ( \| v[\phi] \|_{C^1} + \| v[\psi] \|_{C^1}) . 
\end{split}
\end{equation} We now use Gronwall inequality with $\| v[\phi] \|_{C^1} + \| v[\psi] \|_{C^1} \lec \log N$ and then take  the supremum over $x, x'$ with $x\ne x'$ to obtain \begin{equation}\label{ttp-C-gmm-pre}
\begin{split}
	 \|\ttp\|_{C^\gmm} &= \sup_{x\ne x'}\frac{|\ttp(X(x)) - \ttp(X(x')) |}{ |X(x) - X(x')|^{\gamma} } \\
	 & \le \int_0^t \| H(\cdot, s) \|_{C^\gmm} \exp\left( \int_s^t \left( \| v[\phi](\cdot,\tau) \|_{C^1} + \| v[\psi](\cdot,\tau) \|_{C^1} \right) \d\tau \right) \d s \\
	 & \le\int_0^t  N^{c(t-s)}   \| H(\cdot, s) \|_{C^\gmm} \d s. 
\end{split}
\end{equation} Hence it suffices to estimate $\| H \|_{C^\gamma}$.

We begin by interpolating the $L^\infty$ estimate from Lemma \ref{lem_votp} with the $L^{\infty}$ estimate \eqref{na_votp_Ck} for $\nabla v[\otp]$: \begin{equation*}
		\begin{split}
			 \| v[\otp ] \|_{C^\gmm} \lesssim (\lt \nt)^{\gmm-1} \log N \| \otp \|_{\infty}.
		\end{split}
	\end{equation*} Similarly, we may interpolate \begin{equation*}
	\begin{split}
		 \| v[\ttp ] \|_{C^\gmm} \lesssim  \| v[\ttp ] \|_{\infty}^{1-\gmm} \| \ttp \|_{\infty}^{\gmm} = (\lt \nt)^{\gmm-1} ( \lt \nt  \| W\|_{\infty})^{1-\gmm} \| \ttp \|_{\infty}^{\gmm} .
	\end{split}
	\end{equation*} This gives \begin{equation*}
	\begin{split}
		 \| v[\tp ] \|_{C^\gmm} \lesssim  (\lt \nt)^{\gmm-1} \left( \log N \| \otp \|_{\infty} +  ( \lt \nt  \| W\|_{\infty})^{1-\gmm} \| \ttp \|_{\infty}^{\gmm} \right) 
	\end{split}
	\end{equation*} and we estimate \begin{equation*}
		\begin{split}
			\| II \|_{C^\gmm} & \lesssim \| v[\tp ] \|_{C^\gmm} \| \otp + \tb \|_{C^1} +  \| v[\tp ] \|_{\infty} \| \otp + \tb \|_{C^{1,\gmm}} \\
			&  \lesssim  (\lt \nt)^{\gmm-1} \left( \log N \| \otp \|_{\infty} +  ( \lt \nt  \| W\|_{\infty})^{1-\gmm} \| \ttp \|_{\infty}^{\gmm} \right) \left( (\lt \nt) \| \ttp \|_\infty + \lmb N\log N \right) \\ 
			&\qquad + \left( \| v[\ttp ] \|_\infty +\| v[\otp ] \|_\infty \right) ( (\lt \nt)  \| \otp \|_{C^{\gmm}}  +( \lambda N)^{1+\gmm} \log N ) \\
			& \lesssim  (  (\lt \nt)\| W \|_\infty +  \| \ttp \|_{\infty}  + \log N \| \otp \|_\infty)\left(  (\lt \nt)^{\gmm} \| \ttp \|_{\infty} + \| \otp \|_{C^\gmm} + (\lt \nt)^{-1}(\lambda N)^{1+\gmm}  \log N \right). 
		\end{split}
\end{equation*}
The term $II$ is easily bounded by 
	\begin{equation*}
		\begin{split}
			\| II \|_{C^\gmm}  \lesssim \log N \left(  \| \ttp \|_{C^\gmm} + (\lmb N)^\gmm  \| \ttp \|_\infty\right) .
		\end{split}
	\end{equation*} For $III$, we recall \eqref{III3} that $III_3=\tb \cdot \nabla v[\otp ] $ is the most challenging term (recall the claim \eqref{tb_na_v_otp}). In order to estimate $\| III_3 \|_{C^\gamma}$, we first note that
\[
\| \tb \cdot \nabla v[\otp ]\|_{C^1} \lec_\gamma \| \tb \|_\infty \| \nabla v[ \otp ] \|_{C^1 } + \| \tb \|_{C^1} \| \nabla v[\otp ] \|_\infty \lec \lt^{\frac52 -s} \nt^{1 -s} N^{ct}(\log N )^2 
\]
where we used  \eqref{tb_linfty_Cbeta}, \eqref{tb_Ck} and  \eqref{na_votp_Ck}. Applying H\"older interpolation to this and \eqref{tb_na_v_otp} we thus obtain \begin{equation*}
	\begin{split}
		\| III_3 \|_{C^\gamma} &\lec_\gamma \| III_3 \|_{C^1}^\gamma  \| III_3 \|_{\infty}^{1-\gamma} \lec_\gamma \lt^{\frac32-s} \nt^{-s} N^{ct} (\log N )^2 (\lt^{-1} \lambda N + \nt^{-1} )^{1-\gamma} (\lt \nt )^\gamma \\
		&\lec \lt^{\frac12+2\gamma -s} \nt^{\gamma-s} N^{ct} (\log N )^2  \lambda N 
	\end{split}
\end{equation*}using \eqref{votp_est}. Proceeding similarly as in the $L^\infty$ estimate, we have \begin{equation*}
	\begin{split}		 \| III \|_{C^\gmm} & \lesssim  (\| \ttp \|_\infty + \| \otp \|_\infty )(\| \ttp \|_{C^\gmm} + \| \otp \|_{C^\gmm} )   + \lt^{\frac12+2\gamma -s} \nt^{\gamma-s} N^{ct} (\log N )^2  \lambda N  \\
		 &\qquad + \left( (\lmb N)^{\gmm}  \| \ttp\|_\infty +  \| \ttp\|_{C^\gmm } \right)\log N . 
	\end{split}
	\end{equation*}
Collecting all the bounds, we have \begin{equation*}
	\begin{split}
		& \| H \|_{C^\gmm} \lesssim  (\log N (1 + \|\otp\|_{\infty} ) +  (\lt \nt)\| W \|_\infty  + \| \ttp \|_\infty + \| \otp \|_\infty )(\| \ttp \|_{C^\gmm} + ((\lmb N)^{\gmm} + (\lt\nt)^{\gmm} ) \| \ttp\|_\infty  ) \\
		&\qquad  + \left(  (\lt \nt)\| W \|_\infty  +   \| \ttp \|_\infty + \| \otp \|_\infty \right)   \| \otp \|_{C^\gmm} \\
		&\qquad + (  \lt \nt\| W \|_\infty +  \| \ttp \|_{\infty}  + \log N \| \otp \|_\infty) (\lt \nt)^{-1}(\lambda N)^{1+\gmm}  \log N +  \lt^{\frac12+2\gamma -s} \nt^{\gamma-s} N^{ct} (\log N )^2  \lambda N      . 
	\end{split}
\end{equation*} 
 Using now that $\lt \nt\| W \|_\infty + \| \ttp \|_\infty , \|\otp\|_{\infty} \ll 1$, $\lmb N \ll \lt \nt $,    this estimate simplifies into \begin{equation*}
		\begin{split}
		\| H \|_{C^\gmm}&  \lesssim  \log N (\| \ttp \|_{C^\gmm} + (\lt\nt)^{\gmm} \| \ttp\|_\infty  )  +  (\lt \nt)^{-1-\gmm}(\lambda N)^{1+\gmm} ( \log N)^2  (\lt \nt)^{\gmm}\| \otp \|_\infty   \\
			& \quad + \| \otp \|_{C^\gamma}  + \lt^{\frac12+2\gamma -s} \nt^{\gamma-s} N^{ct} (\log N )^2  \lambda N  \\
			&  \lec \log N (\| \ttp \|_{C^\gmm} + (\lt\nt)^{\gmm} \| \ttp\|_\infty  )  + \lt^{\frac32+\gamma -s} \nt^{\gamma-s} \lmb^{c_{*}} N^{ct} (\log N)^{2}.
		\end{split}
\end{equation*} 
  Then, just using \eqref{ttp_and_W_est} to estimate $ (\lt\nt)^{\gmm} \| \ttp\|_\infty$, we get from \eqref{ttp-C-gmm-pre} and Gronwall's inequality that \begin{equation*}
\begin{split}
	\| \ttp \|_{C^\gmm} \lesssim  \lt^{\frac32+\gamma-s} \nt^{\gamma-s} \lmb^{c_{*}} N^{c t} (\log N)^{3}.
\end{split}
\end{equation*}	 

\medskip

\noindent \textbf{Step 4. Estimate of $\| \ttp \|_{H^{\bt} }$.} 
We will use the equation \eqref{eq_for_W} for $W$, \begin{equation*}
	\begin{split}
		\p_t  W +  (v[\tb ] + v[\psi] )\cdot \nabla W
		= W \cdot \nabla( v[\otp] + v[\tb]  ) - v[\otp] \cdot \nabla v[\otp] + \nabla Q. 
	\end{split}
\end{equation*} Taking the inner product with $W$ and integrating, the pressure term $\nabla Q$ drops after integrating by parts and we immediately obtain \begin{equation*}
\begin{split}
	\frac{\d}{\d t} \norm{W}_{2}^2 \lesssim  \norm{W}_{2}^2\log N + \norm{ \nabla v[\otp] }_\infty \norm{W}_{2} \norm{v[\otp]}_{2}
\end{split}
\end{equation*} which gives after using \eqref{votp-est-L2} \begin{equation*}
\begin{split}
	\frac{\d}{\d t} \norm{W}_{2} & \lesssim  \norm{W}_{2}\log N + \norm{ \otp }_\infty (\lt\nt)^{-1-s} N^{ct} \log N    \lesssim \norm{W}_{2}\log N  + \lt^{\frac12 -2s} \nt^{-1-2s} N^{ct} (\log N)^2 .
\end{split}
\end{equation*} This gives \begin{equation*}
\begin{split}
\norm{W}_{2} \lesssim \lt^{\frac12 -2s} \nt^{-1-2s} N^{ct} (\log N)^2 . 
\end{split}
\end{equation*} To obtain an estimate for the $H^{1}$ norm, we just take a derivative in the equation for $W$ and again take the inner product with itself. Then, the pressure term again disappears and we can estimate the other terms as \begin{equation*}
\begin{split}
		\frac{\d}{\d t} \norm{\nabla W}_{2}^2 &\lesssim  \norm{\nabla W}_{2}^2\log N + \left(\norm{ \nabla v[\otp] }_\infty \norm{\nabla v[\otp]}_{2} + \norm{ D^2 v[\otp] }_\infty \norm{  v[\otp]}_{2} \right) \norm{\nabla W}_{2} \\
		&\qquad + \norm{ D^{2}( v[\otp] + v[\tb]  ) }_{\infty} \norm{W}_{2} \norm{\nabla W}_{2}. 
\end{split}
\end{equation*} This time, proceeding similarly we can obtain \begin{equation*}
\begin{split}
	\norm{\nabla W}_{L^2} \lesssim \lt^{\frac32 -2s} \nt^{-2s} N^{ct} (\log N)^2 .
\end{split}
\end{equation*}  Estimating higher integer order $L^2$ Sobolev norms are analogous, and then the case of general $H^{\bar{\bt}}$ with $\bar{\bt}> 0$ for $W$ follows from interpolation. 
\medskip

\noindent \textbf{Step 5. Conclusion of the proof.}

Steps 1--4 prove the estimates \eqref{W_and_ttp_prop} on time interval $[0,T_0]$. If $B>1$ is sufficiently large (i.e. large enough for $\lambda^{-c_*}$  to be much smaller than $N^{\overline{c}T}$) then, for sufficiently large $\lambda$ these estimates imply that $\| \ttp \|_{C^\gamma} , \| W\|_\infty \leq 1/2$ on $[0,T_0]$. This shows that $T_0=T$, concluding the proof.
\end{proof}

\subsubsection{Proof of Theorem~\ref{thm_black_box}}
Thanks to Proposition~\ref{prop_pert_error} we can now complete the proof of Theorem~\ref{thm_black_box}.

We first  see that the error estimates in Proposition~\ref{prop_pert_error}, the estimates \eqref{otp_Ck} on $\otp$, and the estimate \eqref{votp_est} on $v[\otp]$ give that 
\eqnb\label{temp06}
\begin{split}
\| \tp \|_{C^\gamma} &\lec \lt^{\frac32 + \gamma -s } \nt^{\gamma -s } N^{ct} ,\\
\| \tp \|_\infty &\lec \lt^{\frac32  -s } \nt^{-s } N^{ct} ,\\
\| v[\tp] \|_{\infty} &\leq \| v[\otp] \|_{\infty}+\| W \|_{\infty}\lec \lt^{\frac12  -s } \nt^{-1 -s } N^{ct} \log N 
\end{split}
\eqne
for all $t\in [0,T]$. In particular we can obtain that
\eqnb\label{omega_Ck}
\| \omega \|_{C^k} \lec_k \lambda^{c_k}
\eqne
for all such $t$, and all $k\geq 0$. Indeed, this follows using a similar inductive argument as in the proof of Lemma~\ref{lem_high}. In fact, the induction is simpler, since we only require much less precise bounds \eqref{omega_Ck}, as compared to \eqref{tb_Ck}.  

We now fix $B>1$ to be sufficiently large so that the claims of Lemma~\ref{lem_eta}, Lemma~\ref{lem_votp}, \eqref{otp_Cgamma_shrinks}, \eqref{choice_cstar} and Proposition~\ref{prop_pert_error} are valid and 
\eqnb\label{choice_B}
\| v[\omega ] \|_\infty \leq \| v[\tb ] \|_\infty + \| v[\tp ] \|_\infty \leq c  (\lambda N)^{-1} \log N + c \lt^{\frac12  -s } \nt^{-1 -s } N^{ct} \log N \leq \lambda^{-1} /2T
\eqne
where we used \eqref{vtb_linfty} and \eqref{temp06}. Therefore 
\eqnb\label{tempsupp}
\supp\, \omega(\cdot , t) \subset \supp\, \omega (\cdot ,0 ) +B(\lambda^{-1}/2) \subset B(\lambda^{-1})
\eqne
for all $t\in [0,T]$,  showing the first property in \eqref{black_box_claims}. This and \eqref{omega_Ck} also prove the second property in \eqref{black_box_claims}. As for the third one, we note that 
\[
\| \omega \|_p \lec \lambda^{-\frac3p } \| \omega \|_\infty \lec  \lambda^{-\frac3p } \left(  \| \otb \|_\infty + \|\ttb \|_\infty + \| \otp \|_\infty + \| \ttp \|_\infty \right) \lec K^{-1}\lambda^{-\frac3p }  \log N \lec_p K^{-1} \lambda^{-\frac1p},
\]
as required. The last two claims follow from the $H^\beta$ estimates \eqref{otb_Hb}, \eqref{otp_Hbeta_norm} of the pseudosolutions $\otb$, $\otp$ (respectively), and the $H^\beta$ error estimates \eqref{V_and_ttb}, \eqref{W_and_ttp_prop}. \qedsymbol

\section{Gluing}\label{sec_gluing}
Here we prove the instantaneous continuous loss of regularity, Theorem~\ref{thm_main}.

We first recall the norm inflation Theorem~\ref{thm_black_box}.
\begin{thm}\label{thm_black_box_repeat}
There exists $T\in (0,1/2]$ such that for every $K\geq 1$, and every sufficiently large $\lambda\geq 1$ there exists $\omega^{(0)}_{K,\lambda} \in C_c^\infty $ such that there exists a unique strong solution $\omega_{K,\lambda }$ of the 3D incompressible Euler equations \eqref{3d_euler_vorticity} on time interval $[0,T]$ such that
\[
\omega_{K,\lambda } = \tb + \tp
\]
with 
\eqnb\label{black_box_claims_repeat}
\begin{split}
&\supp\, \omega_{K,\lambda } (\cdot , t)  \subset B(0,\lambda^{-1} )\\
&\| \omega_{K,\lambda } \|_{H^5 } \lec K^{-1} \lambda^{c},\\
&\| \omega_{K,\lambda } \|_{L^p } \lec K^{-1} \lambda^{-c_p}\quad \text{ for every } p\in [1,\infty ),\\
&\| \tb (\cdot , t) \|_{H^\beta }  \sim K^{-1} \lambda^{c_1(\beta -s)}, \quad \text{ uniformly in }t\in [0,T],\\
&\| \tp (\cdot , t) \|_{H^\beta } \sim K^{-t(1+\beta )/s} \lambda^{c_2(\beta - s) + t(1+\beta )c_3}  (\log N )^{-t(1+\beta )/s}
\end{split}
\eqne
where $c_1,c_2,c_3$ are constants depending on $s$ only.
\end{thm}

We let $T\in (0,1/2]$ be given by Theorem~\ref{thm_black_box_repeat} (recall~\eqref{choice_vareps}), and we set 
\eqnb\label{choiceK}
K_j \coloneqq \frac{2^{j}}{\varepsilon}, \quad \text{ for }j\geq 1.
\eqne
We let
\eqnb\label{def_mathcalC}
\mathcal{C} \coloneqq L^2
\eqne
(recall $L$ is the embedding constant from \eqref{Linfty_into_H^2}--\eqref{vLinfty_into_omegaL^4}) and we let $\{\lambda_j\}_{j\geq 1} \subset (4,\infty )$ be increasing sufficiently fast so that  the claim of Theorem~\ref{thm_black_box_repeat} holds and
\eqnb\label{choice_lambdaj}
\lambda_j \geq \exp (K_j),\qquad \| \omega_{K_j,\lambda_j } \|_4 \leq \frac{2^{-(j-1)}}{L}
\eqne
for all $t\in [0,T]$, $j\geq 1$.

We let $\{ D_j \}_{j\geq 1}\subset (2,\infty)$ be a rapidly diverging sequence such that $D_1=3$ and $D_j \geq 2 D_{j-1} +2$ for all $j\geq 2$, to be fixed later.

 For $n\geq 1$ we will use short-hand notation
\[
\omega_{\lambda_n} (\cdot , t) \coloneqq \omega_{K_n,\lambda_n}(\cdot - D_n e_1,t),
\] 
the $\omega_{K_j,\lambda_j}$'s are the solutions to the Euler equations provided by Theorem~\ref{thm_black_box_repeat}, and we will  denote by $\omega^{n}$ the solution to the Euler equation \eqref{3d_euler_vorticity} with initial conditions $\sum_{j=1}^{n}\omega_{\lambda_j}(\cdot ,0)$. We note that $\omega^1$ exists until at least $T_0$ (by Theorem~\ref{thm_black_box_repeat}) and is supported in $B(0,\lambda_1^{-1}) \subset B(0,1/4)$.\\

\noindent\texttt{Step 1.} We show that, if $D_j \to \infty$ sufficiently fast then, for every $n\geq 1$,  $\omega^n$ exists  until at least $T$ and that  and
 \begin{eqnarray}
 \| \omega^n - (\omega^{n-1}+ \omega_{\lambda_n} ) \|_{H^5} &\leq \frac{\varepsilon }{\mathcal{C} 2^n} ,\label{glue1a}\\
 \supp\,\omega^n   &\subset \bigcup_{j=1}^n B(D_j e_1, 1) ,\label{glue1b0}\\
 |\supp\,\omega^n  | &\leq 1 ,\label{glue1b}\\
 \| \omega^n \|_{H^6}  &\leq C_n,\label{glue1c}\\
 \| v[\omega^n ] \|_\infty &\leq 1 ,\label{vel=1}\\
 \|\omega^n  \|_p &\lec_p 1 \label{glue1d}
 \end{eqnarray}
for all $t\in [0,T]$.
 
 The claim is true for $n=1$ by Theorem~\ref{thm_black_box_repeat}. As for $n\geq  2$, note that, by continuity, there exists $T_0\in (0,T]$ such that 
\[
\omega^n = \oo + \wo
\]
for $t\in [0,T_0]$, where $\oo (\cdot , 0) = \sum_{j=1}^{n-1} \omega_{K_j,\lambda_j } (\cdot - D_j e_1,0)$, $\wo (\cdot ,0) = \omega_{K_n,\lambda_n } (\cdot - D_n e_1,0)$ and that 
\eqnb\label{supports_glue}
\supp\,\oo (\cdot ,t )\subset \bigcup_{j= 1}^{n-1} B\left( D_je_1,1 \right), \qquad \supp\,\wo (\cdot , t)\subset B(D_ne_1,1)
\eqne
for $t\in [0,T_0]$. Note that both $\oo$ and $\wo$ satisfy the equation
\[
\p_t \omega + u[ \oo + \wo ] \cdot \nabla \omega = \omega \cdot \nabla u[ \oo + \wo ]
\]  
Setting 
\[
A^n \coloneqq \omega^n - (\omega^{n-1}+ \omega_{\lambda_n} ),\]
for $j\geq 2$, we now show that
\eqnb\label{bound_A}
\| A^n \|_{H^5} \leq \frac{\varepsilon }{2^n \mathcal{C}},
\eqne
where $\mathcal{C}\geq 1$ is defined in \eqref{def_mathcalC} above. We have
\[
\begin{split}
\p_t A^n =& 
- u[\oo  + \wo ] \cdot \nabla \oo + u [\omega^{n-1} ] \cdot \nabla \omega^{n-1} + \oo \cdot u [\oo + \wo ] - \omega^{n-1} \cdot \nabla u [\omega^{n-1}]\\
&- u[\oo  + \wo ] \cdot \nabla \oo + u [\omega_{\lambda_n} ] \cdot \nabla \omega_{\lambda_n} + \oo \cdot u [\oo + \wo ] - \omega_{\lambda_n} \cdot \nabla u [\omega_{\lambda_n}]
\end{split}
\]
Note that the right-hand side above can be decomposed into terms involving only $\omega^{n-1}$ and $\omega_{\lambda_n}$ and
\[
\overline{A^n} \coloneqq \oo - \omega^{n-1},\qquad \widetilde{A^n} \coloneqq \wo - \omega_{\lambda_n}.
\]
Observe that $\overline{A^n}$, $\widetilde{A^n}$ have disjoint supports in space (so that, in particular, $\|\overline{A^n} \|_{H^5} ,\|\widetilde{A^n} \|_{H^5}\leq \| A^n \|_{H^5}$). Let $T'\in (0,T_0]$ be the largest time such that $\| A \|_{H^5} \leq 1$ on $[0,T']$. Then, performing the $H^5$ estimate we obtain
\eqnb\label{est_A}
\begin{split}
\frac{\d }{\d t} \| A^n \|_{H^5} &\leq C_n \| A^n \|_{H^5}^2 +  C_n \left( \| \omega^{n-1} \|_{H^6} + \| \omega_{\lambda_n} \|_{H^6} \right) \| A^n \|_{H^5} + C_n D_n^{-1}\\
& \leq C_n  \| A^n \|_{H^5} + C_n  D_n^{-1}
\end{split}
\eqne
for $t\in [0,T']$, where we noted that the only terms in the $H^5$ estimate not involving $A$ are
\[
u[\omega_{\lambda_n }]\cdot \nabla \omega^{n-1}, \quad u[\omega^{n-1}]\cdot \nabla \omega_{\lambda_n}, \quad \omega^{n-1} \cdot \nabla u [ \omega_{\lambda_n}], \quad \omega_{\lambda_n} \cdot \nabla u [ \omega^{n-1}]
\]
and since the supports (in space) of $\omega^{n-1}$, $\omega_{\lambda_n}$ are separated by at least $D_n-D_{n-1}-2 \gec D_n$, all of the above terms can be bounded in $H^5$ by $C_n \left( \| \omega^{n-1} \|_{H^6} + \| \omega_{\lambda_n} \|_{H^6} \right)^2 D_n^{-1}$. 

From \eqref{est_A} we see that
\[
\| A^n \|_{H^5} \leq C_n D_n^{-1}  \ee^{C_n T_0} \leq D_n^{-1}  \ee^{C_n } 
\]
for all $t\in [0,T']$. Taking $D_n$ sufficiently large the right-hand side is bounded by $\varepsilon / 2^n \mathcal{C}\ll 1$, and so a simple continuity argument shows that \eqref{bound_A} holds for $t\in [0,T_0]$. In order to show that \eqref{bound_A} holds on the entire time interval $[0,T]$, we need to show that $T_0=T$. To this end we will show that \eqref{supports_glue} holds on $[0,T_0]$ with strictly smaller right-hand sides, which implies that $T_0=T$. 

Indeed, since $\lambda_j^{-1} \leq 1/4$ for all $j$ we have that
\eqnb\label{supports_glue_t0}
\supp\,\oo (\cdot ,0) \subset  \bigcup_{j=1}^{n-1} B\left( D_j e_1,\frac14 \right), \qquad \supp\,\wo(\cdot ,0) \subset B\left( D_ne_1,\frac14 \right).
\eqne
Moreover, since 
\[
\omega^n = \sum_{j=1}^n \omega_{\lambda_j} + \sum_{j=2}^n A^j,
\]
we recall \eqref{vLinfty_into_omegaL^4},   and the definition \eqref{def_mathcalC} of $\mathcal{C}$ to obtain
\eqnb\label{vel=1a}
\| v[\omega^n ] \|_\infty \leq L \| \omega^n \|_{L^4}\leq   \frac{1 }2 \sum_{j=1}^n 2^{-j}   + L^2 \sum_{j=2}^n \| A^n \|_{H^5} \leq 1
\eqne
for $t\in [0,T_0]$, where we also used the embedding \eqref{Linfty_into_H^2} as well as our choice \eqref{choice_lambdaj} of the $\lambda_j$'s. Thus, since $T_0\leq T \leq 1/2$, this shows that particle trajectories can travel a distance at most $1/2$. This and \eqref{supports_glue_t0} shows \eqref{supports_glue} for all $t\in [0,T_0]$, which concludes the proof of the existence of $\omega^n$ until $T$ and of \eqref{glue1a}. Moreover, \eqref{glue1b} now follows from the Cauchy formula \eqref{cauchy_formula}, and \eqref{glue1c} follows (with some very large constant $C_n$) by noting that 
\[
\| \omega^n \|_{H^5} \leq \sum_{j=1}^n \| \omega_{\lambda_j} \|_{H^5} + \sum_{j=2}^n \| A^n \|_{H^5} \lec \sum_{j=1}^n \lambda_j^c + \frac{\varepsilon }{\mathcal{C} }  \sum_{j=2}^n 2^{-j} \lec C_n, 
\] 
and using this to perform the $H^6$ estimate on $\omega^n$. Moreover, \eqref{vel=1} follows from \eqref{vel=1a}, and \eqref{glue1d} follows from a computation similar to \eqref{vel=1a},
\[
\| \omega^n \|_p \lec_p \sum_{j\geq 1} \| \omega_{\lambda_j} \|_p + \sum_{j= 2}^n \| A^n \|_{H^5} \lec_p 1
\]
for all $n\geq 1$, where we recalled the third property of \eqref{black_box_claims_repeat} in the last inequality.\\

\noindent\texttt{Step 2.} We take a limit $n\to \infty$ to obtain a solution $\omega$ to the Euler equations \eqref{3d_euler_vorticity} with initial data 
\eqnb\label{initial_data} \omega_0 \coloneqq \sum_{j=1}^\infty \omega_{\lambda_j} (\cdot , 0),
\eqne
such that
\eqnb\label{limit_prop1}
\begin{split}
 \left\| \omega  - \sum_{j\geq 1} \omega_{\lambda_j}  \right\|_{H^5} &\leq  \frac{\varepsilon }{\mathcal{C} } ,\\
 \supp\,\omega   &\subset \bigcup_{j\geq 1} B(D_j e_1, 1) ,\\
 |\supp\,\omega  | &\leq 1 ,\\
 \| \omega \|_{p }  &\leq C_p \qquad \text{ for every }p\in [1,\infty ),\\
 \| v[\omega ] \|_\infty &\leq 1 .
 \end{split}
\eqne

To this end we note that, for every compact $K\subset \R^3$, $\supp\, \omega_{\lambda_m} \cap K = \emptyset$ for sufficiently large $m$, and so \eqref{glue1a} shows that, for each such $m$ and any $n\geq m$,
\[
\| \omega^n - \omega^{m} \|_{H^5(K)} \leq \sum_{j=m+1}^n\| \omega^j - \omega^{j-1} \|_{H^5(K)} \lec \sum_{j\geq {m+1}} 2^{-j} \lec 2^m.
\]
Thus $\{ \omega^n \}_{n\geq 1}$ is Cauchy in $C ([0,T] ; H^5(K))$, and so there exists $\omega\in C ([0,T] ; H^5_{loc}(K))$ such that $\| \omega - \omega^n \|_{C ([0,T] ; H^5(K))} \to 0$ as $n\to \infty$ for each compact $K\subset \R^3$. Recalling the uniform bounds $\| \omega^n \|_1, \| \omega^n \|_4 \lec 1$ (from \eqref{glue1d}), we also obtain a convergence of velocities, i.e. that that $\| v[\omega ] - \omega^n \|_{C ([0,T] ; H^4(K))} \to 0$. In  particular we can take the limit $n\to \infty$ in the weak formulation of the Euler equations \eqref{3d_euler_vorticity} for $\omega^n$ to obtain that $\omega$ also solves the Euler equations on $\R^3\times [0,T]$ in the sense of distributions. The equations imply in particular that $\p_t \omega \in C([0,T]; C^2_{loc})$, so that $\omega \in C^1 ([0,T]; C^2_{loc})$, as required by Definition~\ref{defn_sol}. Taking $n\to \infty $ in the properties \eqref{glue1a}--\eqref{glue1d} implies \eqref{limit_prop1}, concluding this step. \\

\noindent\texttt{Step 3.} We verify instantaneous loss of regularity for $\omega$.\\

Indeed, 
 \[
   \begin{split}
   \|\omega \|_{H^{\beta}}&= \left\|\sum_{j\geq 1}\omega_{\lambda_{j}}\right\|_{H^{\beta}}+ O( \varepsilon )\\
   &= \left( \sum_{j\geq 1}\left\|\omega_{K_{j},\lambda_{j}}\right\|_{H^{\beta}}^2 + O(1)  \right)^{\frac{1}{2}}+O(\varepsilon ),
   \end{split}
   \]
     where we used the first property of \eqref{limit_prop1} in the first equality, and \eqref{slobo_equiv} together with $\sum_{j\geq 1} \| \omega_{\lambda_j} \|_2 \lec 1$ (a~consequence of \eqref{black_box_claims_repeat}) in the second equality.

We now consider $\beta <s$, and note that \eqref{black_box_claims} and the above equality imply that
\eqnb\label{ooo}
\| \omega (\cdot , t ) \|_{H^\beta } < \infty \quad \Leftrightarrow \quad \sum_{j\geq 1} K_{j}^{-t(1+\beta )/s} \lambda_{j}^{c_{3}(\beta-s)+ t(1+\beta ) c_4 } (\log \lambda_j )^{-t (1+ \beta ) /s } <\infty .
\eqne
   Thus, $ \| \omega (\cdot , t ) \|_{H^\beta } < \infty$ iff  $c_{3}(\beta-s)+ t(1+\beta ) c_4 \leq 0$,  since the $\lambda_j$'s dominate the $K_j$'s, due to \eqref{choice_lambdaj}. This gives the loss of regularity claim in Theorem~\ref{thm_main}, since
   \[
   c_{3}(\beta-s)+ t(1+\beta ) c_4 \leq 0 \quad \Leftrightarrow \quad \beta \leq  \frac{s-  \overline{c}t}{1+\overline{c}t},
   \]
   where $\overline{c} \coloneqq c_4/c_3$.\\

\noindent\texttt{Step 4.} We show uniqueness.\\

Suppose that $p\in (3,\infty )$ and that $\wo \in L^\infty_t L^p_x$ is another classical solution (as in Definition~\ref{defn_sol}) of the Euler equation \eqref{3d_euler_vorticity} with initial data \eqref{initial_data}.

We first show that the support of $\wo$ remains localized, that is
\eqnb\label{supp_wo}
\supp \wo \subset \weta (\supp \wo (\cdot ,0),t),
\eqne
where $\weta (y,t)$ denotes the particle trajectories of $v[\wo ]$, i.e. the solution of
\[
\p_t \weta (y,t ) = v[\wo ] (\weta (y,t),t),
\]
with initial condition $\weta (y,0)=y$. Indeed, by Definition~\ref{defn_sol} both $v[\wo ]$ and $\weta $ are well-defined and $C^1_t ([0,T];C^1 (K))$ for any compact set $K\subset \R^3$. For a given $y\in \R^3$ we set  $f(t) \coloneqq \wo (\weta (y,t),t)$, $g(t)\coloneqq v[\wo ] (\weta (y,t),t)$. We have that
\[
f'(t) =f(t)\cdot g(t).
\]
Thus, if $y\not \in \supp \wo (\cdot ,0)$ then $f(0)=0$, and so, since $f\cdot g\in C^1 ([0,T])$, uniqueness of ODE solutions gives that $f(t)=0$ for all $t$. In other words $\wo (x,t)=0$ if $\weta^{-1}(x,t) \not \in \supp\,\wo(\cdot ,0)$, which shows \eqref{supp_wo}.

As a consequence of \eqref{supp_wo} we have that  $|\supp\, \wo (\cdot , t )|\leq 1$ for all $t$, so that, by \eqref{vLinfty_into_omegaL^p} that
\[
\| v[\wo ] \|_{L^\infty} \leq v_{\rm max}
\]
for all $t\in [0,T]$, where $v_{\rm max} >0$.

Let 
\[
W \coloneqq \omega - \wo, \qquad W_j \coloneqq \omega_j - \wo_j,
\]
where $\omega_j(\cdot , t) \coloneqq\left.  \omega (\cdot ,t) \right|_{B(D_je_1,1)}$ and $\wo_j (\cdot ,t ) \coloneqq \left. \wo (\cdot ,t ) \right|_{\weta (B(D_je_1,1),t)}$.

Let $T_0 \in [0,T]$ be the largest time such that $\omega = \wo $ for $t <T_0$. Let $\epsilon = \epsilon (\wo )>0 $ be such that 
\[
\supp\, \wo_j (\cdot , t ) \subset B(D_j e_1, 2) \qquad \text{ for } t\in [0,T_0+\epsilon ]\text{ and all }j,
\]
and that
\eqnb\label{L1_control}
\| W (t) \|_{L^1} \leq 1\qquad \text{ for } t\in [T_0,T_0+\epsilon ].
\eqne
Note that also
\eqnb\label{L1_omega_control}
\| \omega (t) \|_{L^1}  \leq\| \omega (t) \|_{L^4}  \leq c \qquad \text{ for } t\in [T_0,T_0+\epsilon ],
\eqne
for some constant $c>0$ (independent of $\wo$). Note that, for each $j\geq 1$,
\eqnb\label{eq_Wj}
\p_t W_j  + v[\wo ] \cdot \nabla W_j + v [W] \cdot \nabla \omega_j - \omega_j \cdot \nabla u [W] - W_j \cdot \nabla u [\wo ] =0.
\eqne

We will show that $\| W \|_{H^2}=0$ on $[T_0,T_0+\epsilon ]$ (which will then contradict the definition of $T_0$). To this end, we will perform $H^2$ estimates on \eqref{eq_Wj} to obtain an ODE for $\| W \|_{H^2}$. However, we can only use the ODE to obtain an upper bound on $\| W \|_{H^2}$ when we know that $W \in C([T_0,T_0+\epsilon ];H^2)$. This is particularly relevant since a priori there is no reason why $W(t) \in H^2$ (note that neither $\omega$ nor $\wo$ belong to $H^2$). We thus verify this first (in Step 4a below), and then conclude the uniqueness proof in Step 4b.\\

\noindent\texttt{Step 4a.} We show that, if $D_j$ diverges sufficiently fast, then $\| W_j \|_{H^2} \leq 2^{-j}$ for $t\in [T_0 , T_0+\epsilon ]$, $j\geq 1$.\\

Using \eqref{eq_Wj} we get rid of $\wo $ by writing $\wo = \omega - W$ and we decompose both $\omega $ and $W$ into local parts $\omega_j$, $W_j$ and nonlocal parts $\omega_{\ne j}\coloneqq \omega - \omega_j$, $W_{\ne j}\coloneqq W -W_j$, respectively, to obtain
\[
\begin{split}
0 & = \p_t W_j + v [\omega_j  ] \cdot \nabla W_j+ v [\omega_{\ne j}  ] \cdot \nabla W_j + v[W_j  ] \cdot \nabla W_j +v[W_{\ne j}  ] \cdot \nabla W_j  + v[ W_j ] \cdot \nabla \omega_j + v[ W_{\ne j} ] \cdot \nabla \omega_j\\
& \quad  - \omega_j \cdot \nabla v [W_j]- \omega_j \cdot \nabla v [W_{\ne j}] - W_j \cdot \nabla v [\omega_j   ]- W_j \cdot \nabla v [\omega_{\ne j}   ]+ W_j \cdot \nabla v [W_j  ]+ W_j \cdot \nabla v [W_{\ne j}  ].
\end{split}
\]
Rearranging terms into those involving $W_j$ or not, we get
\eqnb\label{rearr}
\begin{split}
0 & = \p_t W_j + \left(\frac{\mbox{}}{\mbox{}} v [\omega_j  ] \cdot \nabla W_j + v [\omega_{\ne j}  ] \cdot \nabla W_j + v[W_j  ] \cdot \nabla W_j +v[W_{\ne j}  ] \cdot \nabla W_j  + v[ W_j ] \cdot \nabla \omega_j - \omega_j \cdot \nabla v [W_j] \right. \\
&\quad \left. - W_j \cdot \nabla v [\omega_j   ]- W_j \cdot \nabla v [\omega_{\ne j}   ]+ W_j \cdot \nabla v [W_j  ]+ W_j \cdot \nabla v [W_{\ne j}  ] \right) \\
&\quad + \left( \frac{\mbox{}}{\mbox{}}v[ W_{\ne j} ] \cdot \nabla \omega_j  - \omega_j \cdot \nabla v [W_{\ne j}] \right)  .
\end{split}
\eqne
We now observe that performing the $H^2$ estimate we obtain, for each term in the first bracket, a bound of the form
\eqnb\label{temp01}
\| W_j \|_{H^2} \left( \frac{\mbox{}}{\mbox{}}\left( \| v[ \omega_j ] \|_{W^{3,\infty}} + \| \omega_j \|_{W^{3,\infty}}  \right) + \left( \| v [\omega_{\ne j} ]\|_{W^{3,\infty}(B(D_je_1,2))}+\| v [W_{\ne j} ]\|_{W^{3,\infty}(B(D_je_1,2))} \right) + \| W_j \|_{H^2}\frac{\mbox{}}{\mbox{}}\right). 
\eqne
Note that the first two terms in the bracket in \eqref{temp01} can be bounded by a constant dependent on $j$ only, since
\[
\| v [\omega_j ] \|_{W^{3,\infty }} \lec 1+ \| \nabla v [\omega_j ] \|_{W^{2,\infty }}\lec 1+ \| \omega_j  \|_{W^{3,\infty }} \leq 1+ \| \omega_{\lambda_j}  \|_{W^{3,\infty }} + \| \omega - \omega_{\lambda_j} \|_{H^5 (B(D_je_1,2))}  \lec C_j
\] 
 by \eqref{vel=1}, \eqref{black_box_claims}, and by bootstrapping \eqref{glue1a}. Moreover, the third and fourth terms in the bracket in \eqref{temp01} can be bounded by a constant multiple of $D_j^{-1}\leq 1$ (due to $D_j\geq 1$ and \eqref{L1_control}--\eqref{L1_omega_control}).  Similarly, the nonlocal terms in \eqref{rearr} (i.e. the  terms in the last bracket in \eqref{rearr}) can be bounded in $H^2$ by $C_j D_j^{-1}$.\\

As for the remaining term in \eqref{temp01}, let $\epsilon_j \in (0,  \epsilon]$ be the largest such $\epsilon_j$ for which  $\| W_j \|_{H^2} \leq 2^{-j}$ for $t\in [T_0,T_0+\epsilon_j]$. Such $\epsilon_j$ exists as $W_j (\cdot ,T_0)=0$ and and $W_j \in C ([T_0,T_0+\epsilon ]; H^2 )$ for each $j$ since the supports of both $\omega_j$ and $\wo_j$ are included in $B(D_je_1,2)$ and $\omega , \wo \in C ([0,T] ; C^2_{loc})$. 

Thus, on $[T_0,T_0+\epsilon_j]$ we obtain
\[
\frac{\d }{\d t} \| W_j \|_{H^2} \leq C_j \left( \| W_j \|_{H^2} + D_j^{-1} \right),
\]
which, together with $ W_j (\cdot ,T_0) =0$ gives that $\| W_j \|_{H^2} \leq 2^{-(j+1)}$ for $t\in [T_0,T_0+\epsilon_j]$, provided $\{ D_j \}$ increases sufficiently fast. This shows that $\epsilon_j = \epsilon$, as otherwise (if $\epsilon_j <\epsilon$) a continuity argument contradicts the definition of $\epsilon_j$. Thus 
\[
\| W_j \|_{H^2} \leq 2^{-j} \qquad \text{ for }t\in [T_0,T_0+\epsilon ],\, j\geq 1,
\]
as requried.\\

\noindent\texttt{Step 4b.} We show that $W(\cdot ,t)=0$ for $t\in [T_0,T_0+\epsilon ]$.\\
(Note that this concludes the uniqueness proof, as it contradicts the definition of $T_0$.)\\

Let 
\[
\delta \coloneqq \max_{[T,T+\epsilon ]} \| W \|_{H^2}.
\]
Note that $\delta \in (0,1]$ by the definition of $T_0$ and by Step 4a.

We perform an $H^2$ estimate of \eqref{temp01} similarly as in Step 4a above, except for the last bracket on the right-hand side of \eqref{temp01}, which we now estimate as
 \[
 \left\|  \frac{\mbox{}}{\mbox{}}v[ W_{\ne j} ] \cdot \nabla \omega_j  - \omega_j \cdot \nabla v [W_{\ne j}] \right\|_{H^2} \lec \| W_{\ne j} \|_{H^2} \| \omega_j \|_{W^{3,\infty } } D_j^{-1} \lec C_j D_j^{-1} \delta ,
 \]
 and so we obtain
 \[
 \frac{\d }{\d t} \| W_j \|_{H^2} \lec C_j \| W_j \|_{H^2} + C_j D_j^{-1} \delta  
 \]
 so that 
 \[
 \| W_j \|_{H^2} \leq C_j D_j^{-1} \delta \ee^{C_j \epsilon}\leq C_j D_j^{-1} \delta \ee^{C_j T }\qquad \text{ for } t\in [T_0,T_0+\epsilon ].
 \]
 If $D_j$ increases sufficiently fast so that the right-hand side is bounded by $\delta/2^j$  (which is a choice independent of $\wo$) we thus have that
 \[
 \| W_j\|_{H^2}^2 \leq \frac{\delta^2}{4^j}
 \]
for $t\in [T_0,T_0+\epsilon ]$. Summing in $j\geq 1$ we have $\| W \|_{H^2}^2 \leq {\delta^2}/3 $
for $t\in [T_0,T_0+\epsilon ]$. Taking $\sup_{[T_0,T_0+\epsilon]}$ gives $\delta^2 \leq \delta^2/3$, a contradiction. \qedsymbol

\section*{Acknowledgements}
LMZ is supported in part by the Spanish Ministry of Science and Innovation by the grant 152878NB-I00 and by the SNF grant FLUTURA:Fluids, Turbulence, Advection No. 212573.
WO was supported by the NSF grant no. DMS-2511556.
IJ was supported by the NRF grant from the Korea government (MSIT), No. 2022R1C1C1011051, RS-2024-00406821. 

\bibliographystyle{siam}
\bibliography{literature}

\end{document}